\let\ORIlabel\label
\let\ORIrefstepcounter\refstepcounter
\AddToHook{package/hyperref/before}{
    \let\label\ORIlabel
    \let\refstepcounter\ORIrefstepcounter
}

\documentclass[onefignum,onetabnum,pagebackref]{siamart220329}

\RemoveFromHook{label}[firstaid/cleveref]

\usepackage{lipsum}
\usepackage{amsfonts}
\usepackage{graphicx}
\usepackage{epstopdf}
\usepackage{tikz}
\usetikzlibrary{positioning}
\ifpdf
  \DeclareGraphicsExtensions{.eps,.pdf,.png,.jpg}
\else
  \DeclareGraphicsExtensions{.eps}
\fi

\usepackage{enumitem}
\setlist[enumerate]{leftmargin=.5in}
\setlist[itemize]{leftmargin=.5in}

\usepackage{amsmath, amssymb, mathtools, stmaryrd}
\usepackage{amsfonts,mathrsfs, mleftright,booktabs,soul}
\usepackage{subcaption}
\usepackage{xspace}
\usepackage{algorithm, algpseudocode, algorithmicx}
\usepackage{backref}

\DeclareFontFamily{U}{mathx}{\hyphenchar\font45}
\DeclareFontShape{U}{mathx}{m}{n}{
      <5> <6> <7> <8> <9> <10>
      <10.95> <12> <14.4> <17.28> <20.74> <24.88>
      mathx10
      }{}
\DeclareSymbolFont{mathx}{U}{mathx}{m}{n}
\DeclareFontSubstitution{U}{mathx}{m}{n}
\DeclareMathAccent{\widecheck}{0}{mathx}{"71}

\newsiamthm{fact}{Fact}
\newsiamthm{question}{Question}
\newsiamthm{remark}{Remark}

\headers{Randomized Jacobi-Davidson}{L. Grigori, T. Park, and I. Simunec}

\title{Randomized Jacobi-Davidson method\thanks{Date: \today
\funding{TP is supported by the RandESC project, funded by the Swiss Platform for Advanced Scientific Computing (PASC).}}}

\author{Laura Grigori\thanks{Institute of Mathematics, EPF Lausanne, 1015 Lausanne, and PSI Center for Scientific Computing, Theory and Data, PSI, 5232 Villigen, Switzerland, (\email{laura.grigori@epfl.ch})} \and Taejun Park\thanks{Institute of Mathematics, EPF Lausanne, 1015 Lausanne, Switzerland, (\email{taejun.park@epfl.ch}, \email{igor.simunec@epfl.ch})} \and Igor Simunec\footnotemark[3]}

\usepackage{amsopn}

\renewcommand*{\backref}[1]{}
\renewcommand*{\backrefalt}[4]{%
	\ifcase #1 %
	(No citations.)
	\or
	(Cited on page #2.)
	\else
	(Cited on pages #2.)
	\fi
}

\usepackage[dvipsnames]{xcolor}

\algrenewcommand\algorithmiccomment[1]{\hfill\textcolor{MidnightBlue}{$\triangleright$ #1}}

\newcommand{\C}{\mathbb{C}}
\newcommand{\abs}[1]{\lvert#1\rvert}
\newcommand{\ip}[2]{\langle#1,\,#2\rangle}

\newcommand{\sep}{\operatorname{sep}}
\newcommand{\Pp}{\mathcal{P}}
\newcommand{\Vv}{\mathcal{V}}
\newcommand{\col}{\operatorname{col}}

\renewcommand{\c}[1]{\widecheck{#1}}
\newcommand{\range}{\operatorname{range}}

\renewcommand{\left}{\mleft}
  \renewcommand{\right}{\mright}

\DeclareMathOperator{\diag}{diag}

\renewcommand{\vec}[1]{\boldsymbol{#1}}

\DeclarePairedDelimiter{\norm}{\lVert}{\rVert}
\DeclareMathOperator{\spn}{span}

\newcommand{\order}{\mathcal{O}}

\DeclareMathOperator*{\argmin}{argmin}

\renewcommand{\hat}[1]{\widehat{#1}}

\newcommand{\gap}{\mathrm{gap}}

\usepackage[sort]{cite}

\begin{document}
	
\maketitle

\begin{abstract}
The Jacobi-Davidson method is a widely used subspace method for computing a few eigenpairs of a large, sparse, non-Hermitian matrix closest to a target. Like other subspace methods, it orthogonalizes each new expansion vector against the whole search basis, at a cost that grows quadratically with the subspace dimension and, in a distributed setting, requires a global synchronization at every iteration. We introduce a randomized Jacobi-Davidson method that replaces this orthogonalization with a much cheaper randomized orthogonalization process, which requires no inner products involving full-dimensional vectors. We prove that, under a uniform separation condition, the randomized method retains the local quadratic convergence of classical Jacobi-Davidson for non-Hermitian problems. The key step is showing that an exact solution of the randomized correction equation is one step of sketched inverse iteration for an arbitrary shift, which lets us extend the analysis to harmonic and refined variants of the sketched extraction, better suited to interior eigenvalues. Numerical experiments on real and synthetic non-Hermitian eigenvalue problems confirm the predicted convergence order and show that the randomized method matches the reliability of classical Jacobi-Davidson while reducing orthogonalization cost.
\end{abstract}
	
	

\section{Introduction}

We consider the problem of computing a few eigenvalues of a large sparse matrix~$A\in\C^{n\times n}$ that lie closest to a prescribed target~$\tau\in\C$, together with the corresponding eigenvectors. Large non-Hermitian eigenvalue problems arise throughout computational sciences including fluid stability~\cite{CliffeSpenceTavener00}, control and model reduction~\cite[Ch.~4]{Antoulas05}, and Markov-chain analysis~\cite[Ch.~1]{Stewart94}. When $n$ is very large, a complete spectral decomposition is typically unfeasible, and the methods of choice are usually \emph{projection methods}, which build an approximation subspace of small dimension by accessing the matrix~$A$ only through matrix-vector products, possibly with a preconditioner; see, for instance, \cite{Saad11}. In this work, we are primarily interested in the non-Hermitian case, which presents more difficulties than the Hermitian one: in general, the eigenvectors are not orthogonal and the associated eigenvalues may be severely ill-conditioned, and an approximation subspace that contains a good approximation to the target eigenvector does not guarantee that the extracted approximate eigenvector is accurate~\cite{JiaStewart01}. This last point is what motivates the harmonic~\cite{Morgan91,MorganZeng98} and refined~\cite{Jia97,JiaStewart01} extractions, which we study in \Cref{sec:harmonicrefined}.

Here we focus on the Jacobi-Davidson method~\cite{SleijpenVanderVorst96}, which combines Davidson's method~\cite{Davidson75} with an orthogonal correction due to Jacobi. Given the current Ritz pair~$(\theta,\vec u)$ and the associated residual, the approximation subspace is expanded by solving a correction equation in which $A-\theta I$ is restricted to the orthogonal complement of~$\vec u$. Under a uniform separation condition, the method converges quadratically for non-Hermitian~$A$ and cubically for Hermitian~$A$~\cite{SleijpenVanderVorst96}. The method has been extended in many directions, including generalized and nonlinear eigenvalue problems; see, for instance, \cite{SleijpenBootenFokkemaVanderVorst96,HochstenbachNotay06} and the references therein.

Similarly to other subspace methods that keep a full basis of the approximation subspace, the Jacobi-Davidson method orthogonalizes each new expansion vector against all the vectors computed so far. After~$m$ iterations, this orthogonalization has a cost of $\order(nm^2)$ floating point operations, which becomes dominant as the subspace dimension $m$ grows, especially if the matrix-vector products with~$A$ are cheap, such as for sparse matrices. On modern architectures, the computational bottleneck is not the number of arithmetic operations, but rather the amount of data movement; more precisely, in a distributed memory setting, each inner product between two vectors requires a global synchronization between processors, which becomes the dominant cost as the number of processors grows. A standard remedy is to periodically restart the method, bounding the subspace dimension by~$m_{\max}$, but this may discard important information and typically slows down convergence.

Randomized numerical linear algebra offers another way to reduce this cost; see~\cite{HalkoMartinssonTropp11,MartinssonTropp20} for an overview. In this context, the main tool is a \emph{subspace embedding}~\cite{Sarlos06}, a matrix $S\in\C^{s\times n}$ with $s\ll n$ that approximately preserves the norm of vectors in a given low-dimensional subspace; such a sketching matrix $S$ can be drawn at random to have this property with high probability for any fixed subspace of a given dimension~\cite{AilonChazelle09,Cohen16,NelsonNguyen13}. 
The randomized Gram--Schmidt process, introduced in~\cite{BalabanovGrigori22} in the context of GMRES, computes a basis that is orthonormal with respect to the sketched inner product $\ip{\vec x}{\vec y}_S := \ip{S\vec x}{S\vec y}$, by computing inner products on vectors of length~$s \ll n$, which are significantly cheaper and involve less communication than inner products of vectors of length $n$. Even though a basis obtained in this way is not orthonormal, it is well-conditioned and it can be reliably used in a (suitably modified) projection method. Randomized orthogonalization has been used successfully for many problems; see \cite{DGST25} for a recent survey.

In this work we propose to use randomized sketching within the Jacobi-Davidson method. We introduce a randomized Jacobi-Davidson method, in which the basis of the search subspace is kept sketch-orthonormal instead of orthonormal, the approximate eigenpairs are extracted from a Galerkin condition in the sketched inner product, and the correction equation is modified by replacing the orthogonal projector onto the complement of the current Ritz vector with the corresponding sketch-orthogonal projector. The resulting algorithm has essentially the same cost per iteration as the original Jacobi-Davidson method, except for the orthogonalization, which is now substantially cheaper. Our analysis shows that the randomized Jacobi-Davidson method still converges quadratically on non-Hermitian problems.

\paragraph{Contributions} Our main contributions are the following:
\begin{itemize}
	\item We introduce the randomized Jacobi-Davidson method (\Cref{alg:rjd}) in \Cref{sec:rjd}, and prove in \cref{thm:convRJD} that it has local quadratic convergence for general non-Hermitian problems, under a uniform separation condition. The randomized basis update requires no length-$n$ inner products: in one iteration at subspace dimension $m$, the classical method performs $12nm$ operations, or $16nm$ if the orthogonalization is repeated for stability, against $6nm + 3T_S$ for the randomized method, where $T_S$ is the cost of sketching a vector and is independent of $m$.
    \item We show in \Cref{thm:ceequalsinviter} that an exact solve of the randomized correction equation is one step of \emph{sketched inverse iteration}, for an arbitrary shift. This is the result that drives the convergence proof and it is what allows the harmonic and refined variants to inherit the analysis unchanged.
	\item We develop harmonic and refined variants of the sketched extraction in \Cref{sec:harmonicrefined}, which are better suited for the approximation of interior eigenvalues. The refined variant can be used to remove the uniform separation condition.
    \item We show in \Cref{remark:normalcase} that for normal $A$ the method converges quadratically, and cubically if the shift is taken to be the standard Rayleigh quotient of the sketched Ritz vector. We nevertheless argue that \Cref{alg:rjd} is not the right tool for Hermitian problems, because sketching destroys the symmetry of the projected problem, resulting in increased cost, whereas the non-Hermitian case is where sketching is free in terms of order of convergence.
\end{itemize}

\paragraph{Related work}
The Jacobi-Davidson method and its variants are the subject of a large literature. Davidson's original method~\cite{Davidson75} was generalized to non-diagonal preconditioners by Morgan and Scott~\cite{MorganScott86}, and the Jacobi-Davidson method~\cite{SleijpenVanderVorst96} replaced the Davidson correction by the projected Newton update~\cite{FokkemaSleijpenVanderVorst98}. Deflated and partial Schur formulations for computing several eigenpairs (JDQR/JDQZ) were developed in \cite{Fokkema1998}, extensions to generalized and polynomial eigenvalue problems in \cite{SleijpenBootenFokkemaVanderVorst96}, and preconditioned and inexact solves of the correction equation in \cite{HochstenbachNotay09}. A mature implementation of these techniques, including thick restarting and several extraction strategies, is available in PRIMME~\cite{PRIMME}. Our contribution is orthogonal to these developments: sketching affects only the inner products used by the method, while the deflation, restarting, and preconditioning strategies remain unchanged.

On the randomized side, sketching has been used to cheaply build well-conditioned bases through randomized Gram-Schmidt~\cite{BalabanovGrigori22} or randomized Householder QR~\cite{doi:10.1137/24M1674327} for various tasks such as for the solution of linear systems and eigenvalue problems~\cite{NakatsukasaTropp24,BalabanovGrigori22, Gubleretal26}, for evaluating matrix functions~\cite{GuttelSchweitzer23,PalittaSchweitzerSimoncini25}, and matrix equations~\cite{PalittaSchweitzerSimoncini24} and in model order reduction~\cite{BalabanovNouy19}. See~\cite{DGST25} for a recent survey. Closest to the present work is the sketched Rayleigh--Ritz procedure of \cite{NakatsukasaTropp24}, and the randomized implicitly restarted Arnoldi method~\cite{deDamasGrigori25}. Combining sketched extraction with Jacobi-Davidson is suggested in \cite{NakatsukasaTropp24}, but is neither developed nor analyzed there. Randomization has
further been used to stabilize the Rayleigh--Ritz extraction itself~\cite{Shao26}, which is complementary to our approach, and the block randomized Gram--Schmidt process of \cite{BalabanovGrigori25} applies when several vectors are added to the subspace at once.

The rest of the paper is organized as follows. In \Cref{sec:preliminaries} we recall the Jacobi-Davidson method and randomized sketching. In \Cref{sec:rjd} we introduce the randomized Jacobi-Davidson method, derive the randomized correction equation, and discuss the computational cost. \Cref{sec:analysis} contains the convergence analysis, which is the main contribution of this work. Harmonic and refined variants of the sketched extraction are developed in \Cref{sec:harmonicrefined}, and \Cref{sec:experiments} reports our numerical experiments. Finally, \Cref{sec:conclusions} contains some concluding remarks.

\paragraph{Notation} 
Throughout, bold lowercase letters denote vectors, $\norm{\cdot}$ denotes the Euclidean norm of a vector and the spectral norm of a matrix, $\ip{\vec x}{\vec y} = \vec x^*\vec y$ the Euclidean inner product, $A^*$ the conjugate transpose and $A^\dagger$ the Moore--Penrose pseudoinverse of $A$. We write $I$ for the identity matrix, $\sigma_{\min}(\cdot)$ and $\sigma_{\max}(\cdot)$ for the smallest and largest singular value, and $\kappa_2(\cdot) = \sigma_{\max}(\cdot)/\sigma_{\min}(\cdot)$ for the spectral condition number. The spectrum of $A$ is denoted by $\lambda(A)$, and for $\mu \in \C$ and $B$ square we denote their separation by $\sep(\mu, B) = \sigma_{\min}(\mu I - B)$. For a matrix $V$ we denote its column space by $\spn(V) = \range(V)$. Given a sketching matrix $S \in \C^{s\times n}$, we write $\ip{\vec x}{\vec y}_S := \ip{S\vec x}{S\vec y}$ for the sketched inner product, $\norm{\vec x}_S := \norm{S\vec x}$ for the induced (semi-)norm, and $\vec x \perp_S \vec y$ for $\ip{\vec x}{\vec y}_S = 0$. A matrix $V$ is sketch-orthonormal (or $S$-orthonormal) if $(SV)^*(SV) = I$. Canonical angles are denoted by $\angle(\cdot,\cdot)$ and are defined in \Cref{def:angles}. Quantities carrying a check, such as $\c{V}_m$ and $\c{\theta}_m$ refer to the classical (non-randomized) Jacobi-Davidson method. Unless stated otherwise, $(\lambda, \vec x_R, \vec x_L)$ is the target eigentriplet.

\section{Preliminaries} 
\label{sec:preliminaries}
This section develops the background needed to describe the randomized Jacobi-Davidson method in \Cref{sec:rjd}. We review the original Jacobi-Davidson method, followed by sketching, a fundamental tool in randomized numerical linear algebra.

\subsection{Jacobi-Davidson method}
The Jacobi-Davidson method introduced by Sleijpen and Van der Vorst \cite{SleijpenVanderVorst96} is a subspace iteration that builds an approximation subspace $\Vv_m = \col(\c{V}_m)$ for the right eigenvector $\vec x_R$ of $A$ associated with the eigenvalue $\lambda$ closest to a target $\tau \in \C$. In each iteration, the Jacobi-Davidson method performs the following four steps: (i) extracts a Ritz pair $(\c{\theta}_m, \c{\vec u}_m = \c{V}_m \c{\vec y}_m)$ from the standard Galerkin condition $\c{V}_m^*(A-\c{\theta}_m I)\c{V}_m \c{\vec y}_m = 0$ on the current subspace; (ii) tests the residual $\c{\vec r}_m = A \c{\vec u}_m - \c{\theta}_m \c{\vec u}_m$ for convergence; (iii) solves the correction equation,
\begin{equation} \label{eq:ce}
    (I-\c{\vec u}_m \c{\vec u}_m^*)(A-\c{\theta}_m I)(I - \c{\vec u}_m \c{\vec u}_m^*)\c{\vec t}_m = -\c{\vec r}_m, \quad \c{\vec t}_m \perp \c{\vec u}_m,
\end{equation} which is the projected Newton update at the current Ritz pair \cite{FokkemaSleijpenVanderVorst98}; and (iv) appends the correction $\c{\vec t}_m$ to the search subspace $\c{V}_m$ via orthogonalization. Under a uniform separation condition, the local convergence is typically quadratic for non-Hermitian $A$ and cubic for Hermitian $A$ \cite{SleijpenVanderVorst96,HochstenbachSleijpen03}. The method is summarized in \Cref{alg:jd}.

The randomized variant described in \Cref{sec:rjd} replaces the standard Galerkin projection $\c{V}_m^*A \c{V}_m$ with a \emph{sketched} Galerkin projection $(SV_m)^*(SAV_m)$ for a fixed sketching matrix $S$ (see \Cref{subsec:sketching}), Euclidean orthogonality by $S$-orthogonality, and modifies the correction equation accordingly. As a result, the orthogonalization and inner products become much more efficient.

\begin{algorithm}[t]
\caption{Jacobi--Davidson method~\cite{SleijpenVanderVorst96}}
\label{alg:jd}
\begin{algorithmic}[1]
\Require Matrix $A\in\C^{n\times n}$; initial orthonormal $V_0\in\C^{n\times m_0}$;
  target $\tau\in\C$; tolerance $\mathtt{tol}>0$; restart dimensions $m_{\min}<m_{\max}$.
\Ensure Approximate eigenpair $(\theta, \vec u)$ with $\|\vec u\|=1$ and
  $\|A \vec u - \theta \vec u\|\leq \mathtt{tol}$.
\State $m\gets m_0$, $V \gets V_0$, $W \gets AV_0$, $H\gets V^*W$ 
\Repeat
  \State Solve $H \vec y = \theta \vec y$ and select the eigenpair closest to $\tau$
  \State Normalize $\vec y \gets \vec y/\norm{\vec y}$ and set $\vec u \gets V \vec y$.
  \State Compute residual: $\vec r \gets W \vec y - \theta \vec u$
  \If{$\|\vec r\|\le \mathtt{tol}$}
  \State Accept $(\theta, \vec u)$ and return
  \EndIf
  \State Solve (approximately) the correction
    equation~\eqref{eq:ce} for $\vec t\perp \vec u$
  \State Orthonormalize: $\hat{\vec t} \gets \vec t - V(V^* \vec t)$, $\hat{\vec t} \gets \hat{\vec t} /\norm{\hat{\vec t} }$
  \State $\vec w \gets A \hat{\vec t}$ \Comment{1 matvec}
  \State Update Gram matrix: $H\gets \begin{bmatrix}
      H & V^*\vec w \\ \hat{\vec t}^*W & \hat{\vec t}^* \vec w
  \end{bmatrix}$
  \State Expand: $V \gets [V,\,\hat{\vec t}]$, $W \gets [W,\, \vec w]$
  \State $m\gets m+1$,
  \If{$m \geq m_{\max}$}
  \State Restart: retain the $m_{\min}$ best Ritz vectors and $m \gets m_{\min}$
  \EndIf
\Until{convergence}
\end{algorithmic}
\end{algorithm}

\subsection{Sketching} \label{subsec:sketching}
Let $\Vv \subset \C^n$ be a fixed subspace of dimension $m$. Then a matrix $S \in \C^{s\times n}$ is a subspace embedding \cite{Sarlos06} for $\Vv$ with distortion $\epsilon \in (0,1)$ if 
\begin{equation} \label{eq:subemb}
    (1-\epsilon)\norm{\vec u}^2 \leq \norm{S\vec u}^2 \leq (1+\epsilon) \norm{\vec u}^2
\end{equation} for all $\vec u \in \Vv$. A subspace embedding reduces the dimension of the space $\C^n$, while preserving the norms of all vectors in the subspace $\Vv$ up to a small distortion factor $(1\pm \epsilon)$. Here, $s$ is called the sketch size and typically, $m < s \ll n$.

In many applications, we do not know the subspace $\Vv$ in advance and hence constructing a subspace embedding can be difficult. Therefore we construct an embedding that satisfies \eqref{eq:subemb} with high probability for any $m$-dimensional subspace in $\C^n$. We refer to these embeddings as \emph{oblivious subspace embeddings}. There are a few important examples of oblivious subspace embeddings which include Gaussian embeddings \cite{DavidsonSzarek01,HalkoMartinssonTropp11}, subsampled randomized trigonometric transforms (SRTTs) \cite{AilonChazelle09,Tropp11}, and sparse embeddings \cite{NelsonNguyen13,Cohen16,cdd26}. These oblivious subspace embeddings have different conditions on $s$ and failure probability for satisfying \eqref{eq:subemb}; see \cite{MartinssonTropp20} and references therein.
In this work, we use a \texttt{sparsestack} embedding~\cite[Definition~1.7]{CEMT25}, which consists of $\zeta = 4$ independent CountSketch matrices stacked vertically, so that $S$ has a total of $\zeta = 4$ nonzeros per column.

Subspace embeddings also approximately preserve the inner products between vectors in the subspace $\Vv$. Let $\ip{\vec u}{\vec v}_S:= \ip{S \vec u}{S \vec v}$ be the sketched inner product, then by \eqref{eq:subemb} and the polarization identity,
\begin{equation} \label{eq:innerproduct}
    |\ip{\vec u}{\vec v}_S - \ip{\vec u}{\vec v}| \leq \epsilon \norm{\vec u}\norm{\vec v},
\end{equation} for all $\vec u,\vec v \in \Vv$. We denote its induced norm restricted to $\Vv$\footnote{Note that without this restriction $\norm{\cdot}_S$ is only a semi-norm.} as $\norm{\vec u}_S := \norm{S \vec u}$. Using the sketched inner product and normalization under $\norm{\cdot}_S$, instead of keeping a basis orthonormal, we can keep it \emph{sketch-orthonormal}. We call a matrix $V$ sketch-orthonormal if $(SV)^*(SV) = I$. By \eqref{eq:subemb}, the singular values of $V$ all lie in $\left[(1+\epsilon)^{-1/2},(1-\epsilon)^{-1/2}\right]$, so $V$ is well-conditioned with its condition number at most $\sqrt{\frac{1+\epsilon}{1-\epsilon}}$. Many algorithms that look for an approximate solution to a linear system or eigenvalue problem in an approximation subspace $\mathcal{V}$, such as Krylov methods, can be modified to employ a well-conditioned sketch-orthonormal basis of $\mathcal{V}$ instead of an orthonormal one. The main advantage of this approach is that a sketch-orthonormal basis can be computed with reduced communication and computational cost using the randomized Gram--Schmidt process; see for instance \cite{BalabanovGrigori22,BalabanovGrigori25,DGST25,PalittaSchweitzerSimoncini25}.

In the context of Jacobi-Davidson, instead of keeping the basis of the search space orthonormal, we will keep it sketch-orthonormal. Therefore, we require $S$ to be a subspace embedding for the sum of the search and image spaces $\col([V_m, AV_m])$, of dimension at most $2m_{\max}$. This can be done using a single sketch by constructing a subspace embedding for subspaces of dimension $2m_{\max}$ at the start of the algorithm, e.g. by taking $s = 4 m_{\max}$. The details are given in \Cref{sec:rjd}. Note that this requirement is an assumption rather than a consequence of obliviousness, since the subspace $\range(V_m)$ produced by the algorithm depends on $S$ through $\theta$. The same difficulty is present in randomized Krylov methods that construct a subspace which depends on $S$, such as restarted Krylov methods \cite{deDamasGrigori25,deDamasGrigoriKS25, NakatsukasaTropp24}.

\section{Randomized Jacobi-Davidson method} 
\label{sec:rjd}
In this section we introduce the randomized Jacobi-Davidson method, which modifies the operations done using standard inner products to ones that use sketched inner products, reducing the orthogonalization cost in the Jacobi-Davidson method. Any orthogonality and normalizations are replaced by their sketched versions. Let ${V}$ be a sketch-orthonormal basis of an approximation subspace $\mathcal{V}\subset \C^n$. Then the algorithm uses \emph{sketched Galerkin projection} $\ip{{V}}{A{V}{\vec y} - {\theta}V{\vec y}}_S = 0$, which leads to the small eigenvalue problem $(S{V})^*(SA{V})\vec {y} = {\theta}{\vec y}$ and the associated approximate eigenpair $({\theta},{\vec u} = {V}{\vec y})$ for $A$. The correction uses the sketch-orthogonal projector $\Pi_{\vec u}^S = I - {\vec u}(S{\vec u})^*S$ in place of the usual orthogonal projector, which can be written as 
\begin{equation} \label{eq:rce}
    \Pi_{\vec u_m}^S(A-{\theta}_mI)\Pi_{\vec u_m}^S{\vec t}_m = -\Pi_{\vec u_m}^S{\vec r}_m, \quad {\vec t}_m \perp_S {\vec u}_m
\end{equation} where ${\vec r}_m = A{\vec u}_m - {\theta}_m{\vec u}_m$ and $\vec x\perp_S \vec y$ means $\ip{\vec x}{\vec y}_S = 0$. Note that $\Pi_{\vec u}^S = I - \vec u (S\vec u)^*S$ is indeed a projector whenever $\norm{\vec u}_S = 1$, and that it is the $S$-orthogonal projector onto the $S$-orthogonal complement of $\vec u$. When $\theta_m$ is the sketched Rayleigh quotient of $\vec u_m$ we have $(S\vec u_m)^*(S\vec r_m) = \theta_m - \theta_m = 0$, so $\Pi_{\vec u_m}^S \vec r_m = \vec r_m$ and \eqref{eq:rce} takes the familiar form with $-\vec r_m$ on the right-hand side. We keep the projector because \eqref{eq:rce} is also used with the harmonic shift in \Cref{sec:harmonicrefined} and with the standard Rayleigh quotient in \Cref{remark:normalcase}; the left-hand side of \eqref{eq:rce} always lies in $\range(\Pi_{\vec u_m}^S)$, so without the projector the equation would be inconsistent. The algorithm is presented in \Cref{alg:rjd}.

\paragraph{Stopping criteria} The natural quantity to monitor is the relative residual of the computed pair, $\norm{A\vec u - \theta \vec u}/\norm{\vec u}$. Since $\vec u = V\vec y \in \spn(V)$ and $S$ is an $\epsilon$-embedding for that subspace, $\norm{\vec u} \geq (1+\epsilon)^{-1/2}\norm{\vec u}_S = (1+\epsilon)^{-1/2}$ and hence $\norm{A\vec u - \theta \vec u}/\norm{\vec u} \leq \sqrt{1+\epsilon} \norm{\vec r}$. Accordingly, we accept the Ritz pair in \Cref{alg:rjd} when $\norm{\vec r} \leq (1+\epsilon)^{-1/2}\mathtt{tol}$. In practice, the distortion $\epsilon$ is unknown, in which case $\norm{\vec r} \leq 2^{-1/2}\mathtt{tol} (\leq (1+\epsilon)^{-1/2} \mathtt{tol})$ can be used instead.

\begin{algorithm}[t]
\caption{Randomized Jacobi--Davidson method}
\label{alg:rjd}
\begin{algorithmic}[1]
\Require Matrix $A\in\C^{n\times n}$; initial $V_0\in\C^{n\times m_0}$; target $\tau\in\C$; tolerance $\mathtt{tol}>0$; restart dimensions $m_{\min}<m_{\max}$.
\Ensure Approximate eigenpair $({\theta}, {\vec u})$ with $\|{\vec u}\|_S=1$ and
  $\|A{\vec u} - {\theta}{\vec u}\|\leq \mathtt{tol} \norm{\vec u}$.
\State Draw an oblivious subspace embedding $S \in \C^{s\times n}$ \Comment{recommended $s = 4m_{\max}$}
\State Sketch-orthonormalize: ${V} \gets \texttt{RGS}(V_0)$ \Comment{\cite[Alg.~2]{BalabanovGrigori22}}
\State $m\gets m_0$, $W \gets A{V}, Q \gets SV, B \gets SW, H\gets Q^*B$
\Repeat
  \State Solve ${H} {\vec y} = {\theta} {\vec y}$ and select the eigenpair closest to $\tau$
  \State Normalize ${\vec y}\gets {\vec y}/\norm{{\vec y}}$ and set ${\vec u} \gets {V} {\vec y}$.
  \State Compute residual: ${\vec r} \gets  W{\vec y} - \theta {\vec u}$
  \If{$\|{\vec r}\|\le (1+\epsilon)^{-1/2}\mathtt{tol}$}
  \State Accept $({\theta}, {\vec u})$ and return
  \EndIf
  \State Solve (approximately) the randomized correction
    equation \cref{eq:rce},
    \begin{equation*}
        (I - {\vec u}(S{\vec u})^*S)(A-{\theta}I)(I - {\vec u}(S{\vec u})^*S){\vec t} = - {\vec r}, \quad {\vec t} \perp_S{\vec u}.
    \end{equation*}
  \State Sketch-orthonormalize: $\hat{\vec t} \gets {\vec t} - {V}Q^*S\vec t$, $\widehat{\vec q} \gets S\widehat{\vec t}$, $\hat{\vec t} \gets \hat{\vec t} /\norm{\widehat{\vec q}}$, $\widehat{\vec q} \gets \widehat{\vec q} /\norm{\widehat{\vec q}}$
  \State $\vec w \gets A\hat {\vec t}$  \Comment{1 matvec}
  \State Sketch: $\vec b \gets S\vec w$
  \State Update Gram matrix: $H \gets \begin{bmatrix}
      H & Q^*\vec b \\ \widehat{\vec q}^*B & \widehat{\vec q}^* \vec b
  \end{bmatrix}$
  \State Expand: ${V} \gets [{V},\,\hat {\vec t}]$, $W \gets [W,\, \vec w]$, $Q\gets [Q,\,\widehat{\vec q}]$, $B \gets [B,\, \vec b]$
  \State $m\gets m+1$,
  \If{$m \geq m_{\max}$}
  \State Restart: retain the $m_{\min}$ best Ritz vectors and $m\gets m_{\min}$
  \EndIf
\Until{convergence}
\end{algorithmic}
\end{algorithm}

\subsection{Derivation of the randomized correction equation}
The randomized correction equation \eqref{eq:rce} is the natural randomized analogue of the classical Jacobi-Davidson correction equation that incorporates sketched orthogonality. Given the current sketched Ritz pair $({\theta},{\vec u})$ with $\norm{{\vec u}}_S = 1$ and residual ${\vec r} = A{\vec u} - {\theta}{\vec u}$, we seek a correction ${\vec t} \perp_S {\vec u}$ such that ${\vec u} +{\vec t}$ is closer to an eigenvector. When the correction equation is solved exactly, the resulting iterates coincide with those generated by inverse iteration with the shift $\theta$ and a sketched normalization. We refer to this method as \emph{sketched inverse iteration}. The result below holds for an arbitrary shift $\theta$, not only for the sketched Rayleigh quotient; this generality is used in \Cref{remark:normalcase} and \Cref{sec:harmonicrefined}, where $\theta$ is a standard Rayleigh quotient or a harmonic Ritz value, respectively.

\begin{theorem}[Sketched inverse iteration] \label{thm:ceequalsinviter}
    Let ${\vec u}$ be a vector such that $\norm{{\vec u}}_S = 1$, let $\theta \in \C$ with $\theta \notin \lambda(A)$, and assume $(S{\vec u})^*S(A-{\theta}I)^{-1}{\vec u} \neq 0$. Then there exists a unique solution ${\vec t} \perp_S {\vec u}$ of \eqref{eq:rce} that satisfies
    \begin{equation}
		\label{eq:correction-equation-update}
        {\vec u} + {\vec t} = {\alpha}(A-{\theta}I)^{-1}{\vec u},
    \end{equation} where ${\alpha} = ((S{\vec u})^*S(A-{\theta}I)^{-1}{\vec u})^{-1}$. In particular this applies to the sketched Rayleigh quotient $\theta = (S\vec u)^*(SA\vec u)$ used in \Cref{alg:rjd}.
\end{theorem}
\begin{proof}
    The proof follows the deterministic case \cite{SleijpenVanderVorst96} with the sketch-orthogonal projector $\Pi_{\vec u}^S = I - {\vec u}(S{\vec u})^*S$. From \eqref{eq:rce} and ${\vec t}\perp_S {\vec u}$, which gives $\Pi_{\vec u}^S \vec t = \vec t$, we have 
    \begin{equation*}
        (A-\theta I)\vec t - \vec u (S\vec u)^*S(A-\theta I)\vec t = -(A-\theta I)\vec u + \vec u (S\vec u)^*S\vec r.
    \end{equation*}
    Setting $\alpha = (S\vec u)^*S(A-\theta I)\vec t +(S\vec u)^*S\vec r$, gives $(A-\theta I)(\vec u+\vec t) = \alpha \vec u$. To obtain the expression for $\alpha$, apply $(S{\vec u})^*S$ to both sides of \cref{eq:correction-equation-update} and use $(S\vec u)^*(S\vec u) = 1$ and $(S{\vec u})^*(S{\vec t}) = 0$. When $\theta$ is the sketched Rayleigh quotient, $(S\vec u)^*(S\vec r) = 0$ and $\alpha$ reduces to $(S\vec u)^*(SA\vec t)$, similarly to the deterministic case. To show existence, note that if $\vec t = \alpha(A-\theta I)^{-1}\vec u -\vec u$ such that $\vec t \perp_S \vec u$ then we have
    \begin{align*}
        \Pi_{\vec u}^S (A-\theta I) \Pi_{\vec u}^S \vec t =  \Pi_{\vec u}^S (A-\theta I)\vec t = \Pi_{\vec u}^S (A-\theta I) (\alpha(A-\theta I)^{-1}\vec u -\vec u) = -\Pi_{\vec u}^S  \vec r.
    \end{align*} Uniqueness of the solution to \eqref{eq:rce} follows directly from the expression for $\vec t$ in \eqref{eq:correction-equation-update}.
\end{proof}

\subsection{Computational cost}
We compare the computational cost of one iteration of \Cref{alg:jd} and \Cref{alg:rjd} at subspace dimension $m$, counting only the dominant terms. We exclude the matrix-vector product with $A$, the solve of the correction equation and the $\order(m^3)$ solve of the projected eigenvalue problem, all of which we assume identical for the two methods.
\begin{itemize}
    \item \emph{Classical Jacobi-Davidson}. Updating $H = V^*W$ requires computing $V^*\vec w$ and $\widehat{\vec{t}}^*W$, at a cost of $4nm$ operations; orthogonalizing the new direction against $V$ costs $4nm$ operations with Gram-Schmidt, and twice that if it is repeated for stability. Forming $\vec u = V\vec y$ and $\vec r = W\vec y - \theta \vec u$ costs a further $4nm$. The total is $12nm$ operations, or $16nm$ if orthogonalization is repeated for stability. Over $m$ iterations, the total cost is $\order(nm^2)$.
    \item \emph{Randomized Jacobi-Davidson}. In the randomized method, the inner products are computed in the smaller sketched space, so they cost only $\order(sm)$, where $s \ll n$. The update of $H$ costs $4sm$ flops and the projection coefficients $Q^*S\vec t$ in the randomized orthogonalization step cost $2sm$. The main operation involving vectors of length $n$ is the update $\widehat{\vec t} \gets \vec t - V(Q^*S\vec t)$, which costs about $2nm$ operations. Forming the Ritz vector and residual cost another $4nm$. Sketching the three new vectors $\vec w$, $\vec t$, and $\widehat{\vec t}$ \footnote{The third sketch is optional: $S\widehat{\vec t}$ can also be obtained from quantities that are already available, as $S\widehat{\vec t} = S\vec t - Q(Q^*S\vec t)$, which saves one sketching cost per iteration. We prefer to sketch $\widehat{\vec t}$ explicitly, because this keeps $Q = SV$ to working accuracy.} costs $3T_S$ where $T_S$ is the cost of applying the sketching matrix to a vector, e.g., $\order(n\log n)$ for an SRTT, or $\order(\zeta n)$ for a sparse embedding with $\zeta$ nonzeros per column. Thus, when $s\ll n$, the dominant cost is approximately $6nm + 3T_S$ per iteration. Over the full run, the total cost is again $\order(nm^2)$.
\end{itemize}
Both methods therefore have the same asymptotic complexity, but the classical one performs roughly twice as much arithmetic operations involving length-$n$ vectors. The saving comes from replacing length-$n$ inner products with length-$s$ inner products via sketching. The more important saving is in communication. In a distributed-memory setting, with the long vectors distributed across processors, the classical method performs $\order(1)$ global reductions on length-$n$ vectors per iteration for the Gram-Schmidt step and $\order(1)$ more for the update of $H$. Each such reduction introduces a synchronization point. In contrast, the randomized method performs the corresponding operations on the sketched vectors, which have length $s \ll n$ and can be replicated on every processor. However, the application of $S$ requires communication, making the communication cost of randomized Gram-Schmidt the same as classical Gram-Schmidt (CGS), but cheaper than modified Gram-Schmidt (MGS) or CGS with reorthogonalization (CGS2) at a similar stability as MGS~\cite[Sec.~2.4]{BalabanovGrigori22}. The trade-off is a sketching distortion $\epsilon$, whose effect on convergence is analyzed in \Cref{sec:analysis}. 

 \begin{remark}[Inexact solves and mixed precision]
    In practice, \eqref{eq:rce} is solved approximately using a few steps of a (preconditioned) Krylov method. We report the effect of a fixed inner iteration budget in \Cref{subsec:inexactce}. Sketching also makes the method a natural candidate for mixed-precision implementation \cite{BalabanovGrigori22}, since the sketched quantities need only be as accurate as the distortion $\epsilon$; this is done for the Hermitian case in~\cite{Gubleretal26}, but we do not pursue it here.  
 \end{remark}

\section{Convergence analysis} \label{sec:analysis}
We analyze the convergence of \Cref{alg:rjd} without restarting, under the assumption that the randomized correction equation \eqref{eq:rce} is solved exactly at each step. The argument combines three ingredients: (i) a bound on the sketched Ritz value; (ii) a bound on the sketched Ritz vector under a uniform separation condition; and (iii) the contraction produced by one exact step of the randomized correction equation, which is a single step of sketched inverse iteration.

Throughout, let $(\lambda, \vec x_R, \vec x_L)$ be a simple eigentriplet of $A$ normalized so that $\norm{\vec x_R} = \norm{\vec x_L} = 1$, where $A\vec x_R = \lambda \vec x_R$ and $\vec x_L^*A = \lambda \vec x_L^*$. Since $\lambda$ is simple, $\vec x_L^*\vec x_R \neq 0$, and the condition number of $\lambda$ is $\kappa(\lambda) = |\vec x_L^*\vec x_R|^{-1} \geq 1$. Let $\Pp = \frac{\vec x_R \vec x_L^*}{\vec x_L^*\vec x_R}$ be the spectral projector onto $\spn(\vec x_R)$, and $X_\perp \in \C^{n\times (n-1)}$ be a matrix with orthonormal columns spanning the complementary invariant subspace $\range(I - \Pp) = \ker(\Pp) = \spn(\vec x_L)^\perp$. Since $\Pp$ commutes with $A$, both $\range(\Pp)$ and $\range(I-\Pp)$ are $A$-invariant. Consequently, $[\vec x_R, X_\perp]$ is nonsingular and $A = [\vec x_R, X_\perp] \begin{bmatrix}
    \lambda & 0 \\ 0 & A_{22}
\end{bmatrix} [\vec x_R,X_\perp]^{-1}$. Here, $A_{22}$ is the deflated block of $A$ at $\lambda$ and represents the restriction of $A$ to the complementary invariant subspace $\range(I-\Pp)$ in the basis $X_\perp$; equivalently $AX_\perp = X_\perp A_{22}$ and $A_{22} = X_{\perp}^*AX_{\perp}$. A short computation gives $\kappa_2([\vec x_R,X_\perp]) = \kappa(\lambda)+\sqrt{\kappa(\lambda)^2-1}$, which is at most $2\kappa(\lambda)$, so the deflation basis is well-conditioned precisely when $\lambda$ is. Finally, define $\eta:=\sep(\lambda,A_{22}) = \sigma_{\min}(\lambda I - A_{22})$. Because $\lambda \neq \lambda(A_{22})$, the matrix $\lambda I - A_{22}$ is nonsingular, and hence $\eta >0$.

Throughout the analysis, if $\theta_m = \lambda$ and $\vec r_m = 0$, the exact eigenpair has been obtained. If $\theta_m = \lambda$, $\vec r_m \neq 0$, and the projected correction equation is solvable, then its exact solution produces $\vec u_m + \vec t_m \in \spn(\vec x_R)$, so convergence occurs in one further expansion. We therefore treat below the remaining case $\theta_m \neq \lambda$.

Canonical angles \cite{StewartSun90Book} are a natural measure to quantify the distance between two vectors or subspaces. Before beginning the convergence analysis, we define the angle notions the proofs rely on.

\begin{definition}[Canonical angles] \label{def:angles}
    For a nonzero vector $\vec x$ and an orthonormal matrix $W$, the canonical angle is measured by
    \begin{equation*}
        \sin\angle(\vec x,W) = \frac{\min\limits_{\vec w\in \spn(W)}\norm{\vec x-\vec w}}{\norm{\vec x}}, \quad \cos\angle(\vec x,W) = \frac{\norm{\Pp_W \vec x}}{\norm{\vec x}},
    \end{equation*} where $\Pp_W = WW^*$ is the orthogonal projector onto $\spn(W)$.
\end{definition}

Via canonical angles, we can always decompose a vector in components that are respectively parallel and orthogonal to a subspace, using sine and cosine of canonical angles. For example, if $\vec x$ is a unit vector and $V$ is an orthonormal matrix, $\vec x$ can be decomposed as $\vec x = c \, \vec v + s \, \vec v_\perp$, where $c = \cos\angle(\vec x,V)$, $s = \sin\angle(\vec x,V)$ and $\vec v \in \spn(V)$ and $\vec v_\perp \in \spn(V)^\perp$ with $\norm{\vec v} = \norm{\vec v_{\perp}} = 1$. The vectors $\vec v$ and $\vec v_\perp$ are, respectively, the normalizations of $VV^*\vec x \in \spn(V)$ and $(I - VV^*)\vec x \in \spn(V)^\perp$. This construction will be used repeatedly in the proofs that follow.

We measure convergence of the search space by $\rho_m:= \sin\angle(\vec x_R, V_m)$ and of the Ritz vector by $s_m:= \sin\angle(\vec x_R, {\vec u}_m)$. We also use the notation $c_m := \cos \angle(\vec x_R, \vec u_m)$. 


\subsection{The sketched Ritz value} For standard Galerkin extraction, the target Ritz value equals the Rayleigh quotient of the Ritz vector, that is, 
\begin{equation*}
	\c{\theta}_m = \frac{\c{\vec u}_m^*A\c{\vec u}_m}{\c{\vec u}_m^*\c{\vec u}_m} = \frac{\ip{\c{\vec u}_m}{A\c{\vec u}_m}}{\ip{\c{\vec u}_m}{\c{\vec u}_m}}.
\end{equation*}
For the randomized Jacobi-Davidson method, we use sketched Galerkin extraction. Letting $H = (SV_m)^*(SAV_m)$ and writing ${\vec u}_m = {V}_m {\vec y}$ with ${H}{\vec y} = {\theta}_m {\vec y}$ and $V_m$ sketch-orthogonal, we have
\begin{equation}
    {\theta}_m = \frac{(S{\vec u}_m)^*(SA{\vec u}_m)}{(S{\vec u}_m)^*(S{\vec u}_m)} = \frac{\ip{{\vec u}_m}{A{\vec u}_m}_S}{\ip{{\vec u}_m}{{\vec u}_m}_S}.
\end{equation} When we normalize ${\vec y}$ so that $\norm{{\vec y}} = 1$, then ${\theta}_m = \ip{{\vec u}_m}{A{\vec u}_m}_S$. 

The next bound makes precise the size of $|{\theta}_m-\lambda|$ in terms of the Ritz vector angle $s_m = \sin\angle(\vec x_R, {\vec u}_m)$.

\begin{lemma}[Sketched Ritz value error] \label{lem:sketchedRitzvalbnd}
    Let $\nu(\lambda):= \norm{(A^*-\bar{\lambda}I) \vec x_R}$ denote the adjoint residual norm, $s_m = \sin\angle(\vec x_R, {\vec u}_m)$, and $c_m = \cos\angle(\vec x_R, {\vec u}_m)$. If $S$ is an $\epsilon$-subspace embedding for $\spn([{V}_m, A{V}_m])$ and $\norm{\vec u_m}_S = 1$, then
    \begin{equation}
        \abs{{\theta}_m - \lambda} \leq \frac{s_m}{1-\epsilon}\left(c_m \nu(\lambda) + (s_m+\epsilon) \norm{A-\lambda I}\right).
    \end{equation}
    When the left and right eigenvectors coincide, e.g., when $A$ is normal, $\nu(\lambda) = 0$ and $\abs{{\theta}_m - \lambda} \leq \frac{s_m(s_m+\epsilon)}{1-\epsilon} \norm{A-\lambda I}$.
\end{lemma}
\begin{proof}
    Write ${\vec u}_m = \norm{{\vec u}_m}(c_m  \vec x_R + s_m \vec v_m)$ with $\vec v_m \perp  \vec x_R$ and $\norm{\vec v_m} = 1$. Since $(A-\lambda I) \vec x_R = 0$, $(A-\lambda I) {\vec u}_m = \norm{{\vec u}_m} s_m (A-\lambda I)\vec v_m$. Noting $\norm{{\vec u}_m}_S = 1$, hence $\norm{\vec u_m}^2 \leq (1-\epsilon)^{-1}$, and expanding gives
    \begin{align*}
        \abs{{\theta}_m - \lambda}
        &= \abs{\ip{{\vec u}_m}{(A-\lambda I){\vec u}_m}_S} \leq  \abs{\ip{{\vec u}_m}{(A-\lambda I){\vec u}_m}} + \epsilon \norm{{\vec u}_m}\norm{(A-\lambda I){\vec u}_m} \\
        &\leq s_m \norm{{\vec u}_m}^2\abs{\ip{(A^*-\bar{\lambda}I)(c_m \vec x_R + s_m \vec v_m)}{\vec v_m}} + \epsilon \norm{{\vec u}_m}^2 s_m \norm{A-\lambda I} \\
        &\leq \frac{s_m}{1-\epsilon}\left(c_m \nu(\lambda) + (s_m+\epsilon) \norm{A-\lambda I}\right),
    \end{align*} where we used \eqref{eq:innerproduct} for the first inequality and \eqref{eq:subemb} for the final inequality.
\end{proof}


\subsection{The sketched Ritz vector}
For non-Hermitian $A$, the Rayleigh-Ritz vector need not converge merely because the subspace does; one needs a \emph{uniform separation condition} \cite{JiaStewart01} ensuring the target Ritz value stays isolated in the projected spectrum. 

The analysis of this section establishes convergence of the sketched Galerkin extraction to a target eigenpair $(\lambda,\vec x_R)$ under the uniform separation condition associated with the sketched inner product; see \Cref{def:uniformseparation}. This condition is a property of the projected pencil rather than of the location of $\lambda$ in the spectrum of $A$. Nevertheless, the spectral location of the target strongly affects how likely the condition is to hold. Standard Rayleigh--Ritz extraction often favors eigenvalues near the exterior of the spectral region, whereas, for an interior target, an unwanted or spurious eigenvalue of the projected problem may approach $\lambda$ as the trial subspace grows. In that case the separation constant may become small, and the Ritz vector may fail to converge even when the trial subspace contains an increasingly accurate approximation to $\vec x_R$ \cite{Jia97,JiaStewart01,MorganZeng98}. For Hermitian problems, the favorable behavior of extremal Ritz values is supported by the Cauchy interlacing theorem~\cite[Ch.~4]{Saad11}; for non-Hermitian problems, the exterior--interior distinction is instead a useful heuristic and does not by itself imply uniform separation. The results of this section therefore apply most naturally when the target remains uniformly separated in the projected spectrum, a situation often encountered for exterior eigenvalues. The harmonic variant in \Cref{subsec:harmonic} is designed for interior targets by applying the extraction to a shift-transformed problem in which eigenvalues near the shift become dominant, although an analogous separation condition is still required \cite{Morgan91,MorganZeng98}. The refined variant in \Cref{subsec:refined} instead minimizes the residual over the trial subspace and removes the need for separation from the remaining projected eigenvalues. For a simple target eigenvalue, the refined vector converges provided that the trial subspace approaches $\vec x_R$ and the accompanying eigenvalue approximation converges to $\lambda$~\cite{Jia97,JiaStewart01,Jia05}. Finally, Shao~\cite{Shao26} develops a randomized Rayleigh--Ritz procedure that, for a simple target eigenpair, extracts with high probability an approximation whose convergence rate is comparable to that of the trial subspace; we do not pursue this alternative in the present work.

Let $S$ be an $\epsilon$-subspace embedding for $\mathcal{W} = \spn([{V}_m,A{V}_m])$. Then $\ip{\vec u}{\vec v}_S = \ip{S\vec u}{S\vec v}$ is an inner product on $\mathcal{W}$. Suppose that $({\theta}_m,{\vec u}_m = {V_m}\vec y_m)$ is a Ritz pair, i.e., ${H}_m \vec y_m = {\theta}_m \vec y_m$ with $\norm{\vec y_m} = 1$. Now extend $\vec y_m$ to a unitary matrix $Y = [\vec y_m,Y_\perp]$ such that $\vec y_m^*Y_\perp = 0$, then $Y^*{H}_mY$ can be partitioned as
\begin{equation} \label{eq:partitionUSC}
    Y^*{H}_mY = [\vec y_m,Y_\perp]^*{H}_m [\vec y_m,Y_\perp] = \begin{bmatrix}
        {\theta}_m & H_{12} \\ 0 & H_{22}
    \end{bmatrix}
\end{equation} where $H_{12} = \vec y_m^*{H}_mY_\perp$ and $H_{22} = Y_\perp^*{H}_m Y_\perp$. The uniform separation condition defines how well the target eigenvalue $\lambda$ of the original matrix $A$ is separated from the non-targeted eigenvalues of ${H}_m$, i.e., the eigenvalues of $H_{22}$. This is defined below in \Cref{def:uniformseparation}.

\begin{definition}[Uniform separation condition] \label{def:uniformseparation}
    In the above setting, partition $Y^*{H}_m Y$ as in \eqref{eq:partitionUSC}. Suppose that the Ritz pair $({\theta}_m,{\vec u}_m = {V}_m \vec y_m)$ targets the eigenpair $(\lambda,\vec x_R)$ of $A$. Then the iteration satisfies the uniform separation condition with constant $\beta>0$ if $\sigma_{\min} (H_{22} - \lambda I) \geq \beta$ for all $m$, i.e., $\lambda$ is separated from the spectrum of $H_{22}$.
\end{definition}

\begin{lemma}[Sketched Ritz vector bound] \label{lem:sketchedRitzvecbnd}
Suppose $S$ is an $\epsilon$-subspace embedding for $\spn([{V}_m, A{V}_m])$ and that the uniform separation condition (\Cref{def:uniformseparation}) holds with constant $\beta$. Then, letting $K_\epsilon := \frac{1+\epsilon}{1-\epsilon}$, we have
\begin{equation}
    \sin\angle(\vec x_R, {\vec u}_m) \leq  \sqrt{1+\frac{K_\epsilon \norm{A-\lambda I}^2}{\beta^2}} \sin\angle(\vec x_R, {V}_m).
\end{equation} 
\end{lemma}
\begin{proof}
    Decompose $\vec x_R = {V}_m \vec c + V_{\perp} \vec x_3$ where $V_\perp$ is an orthonormal basis of $\spn({V}_m)^\perp$, $\vec c = {V}_m^\dagger \vec x_R$, and $\vec x_3 = V_\perp^* \vec x_R$. Note that $\norm{\vec x_3} = \sin\angle(\vec x_R, {V}_m)=:\rho_m$. Now extend $\vec y_m$ to a unitary matrix $Y = [\vec y_m,Y_\perp]$ such that $\vec y_m^*Y_\perp = 0$. We can write $\vec c = x_1 \vec y_m + Y_\perp \vec x_2$ where $x_1 = \vec y_m^*\vec c$ and $\vec x_2 = Y_\perp^*\vec c$.

    Using this setup and $\norm{{V}_m} \leq (1-\epsilon)^{-1/2}$, we obtain
    \begin{align*}
        \sin^2\angle(\vec x_R, {\vec u}_m) = \min_t \frac{\norm{\vec x_R - t{\vec u}_m}^2}{\norm{\vec x_R}^2} &= \min_t \norm{{V}_m \vec c+V_\perp \vec x_3 - t {V}_m \vec y_m}^2 \\
        &= \min_t \norm{{V}_m (x_1 \vec y_m+Y_\perp \vec x_2 - t\vec y_m)}^2 +\norm{V_\perp \vec x_3}^2 \\
        &\leq \norm{{V}_m}^2 \min_t  \norm{(x_1 - t)\vec y_m +Y_\perp \vec x_2}^2 + \rho_m^2 \\
        &\leq \frac{1}{1-\epsilon} \norm{Y_\perp \vec x_2}^2 +\rho_m^2, 
    \end{align*} where we set $t = x_1$ for the last inequality. 
    
    Now we bound $\norm{Y_\perp \vec x_2} = \norm{\vec x_2}$. We have
    \begin{align*}
        {H}_m \vec c = (S{V}_m)^*(SA{V}_m{V}_m^\dagger \vec x_R) &= \lambda (S{V}_m)^*S\vec x_R - (S{V}_m)^* SA(I - {V}_m{V}_m^\dagger)\vec x_R \\
        &= \lambda (S{V}_m)^*S({V}_{m}\vec c +V_\perp \vec x_3) - (S{V}_m)^* SA(I - {V}_m{V}_m^\dagger)\vec x_R \\
        &= \lambda \vec c - (S{V}_m)^* S(A-\lambda I)V_\perp V_\perp^* \vec x_R.
    \end{align*} Now multiply $Y_\perp^*$ on the left and using $Y_\perp^*{H}_m \vec y_m = 0$ and $Y_\perp ^*\vec y_m = 0$ to obtain,
    \begin{equation*}
        Y_\perp^*({H}_m-\lambda I) \vec c = Y_\perp^*({H}_m - \lambda I) (\vec y_m x_1 + Y_\perp \vec x_2) = Y_\perp^*({H}_m - \lambda I) Y_\perp \vec x_2 = (H_{22} - \lambda I) \vec x_2.
    \end{equation*}
    Therefore by using the uniform separation condition we get
    \begin{align*}
        \norm{\vec x_2} = \norm{(H_{22} - \lambda I)^{-1} Y_\perp^*({H}_m - \lambda I)\vec c} &\leq \frac{\norm{({H}_m - \lambda I)\vec c}}{\beta} \\
        &= \frac{\norm{(S{V}_m)^* S(A-\lambda I)V_\perp V_\perp^* \vec x_R}}{\beta} \\
        &\leq \frac{\norm{S(A-\lambda I)V_\perp V_\perp^* \vec x_R}}{\beta} \\
        &\leq \frac{\sqrt{1+\epsilon}\norm{(A-\lambda I)V_\perp V_\perp^* \vec x_R}}{\beta} \\
        &\leq \frac{\sqrt{1+\epsilon}\norm{A-\lambda I}\rho_m}{\beta},
    \end{align*} where for the penultimate inequality we use the fact that $S$ is a subspace embedding on $\spn([{V}_m,A{V}_m])$ as $(A-\lambda I)V_\perp V_\perp^* \vec x_R =  (A-\lambda I)(I-{V}_m{V}_m^\dagger)\vec x_R = -(A-\lambda I)V_m V_m^\dagger \vec x_R  \in \spn(A{V}_m) + \spn({V}_m)$, where we used $(A-\lambda I)\vec x_R = 0$.
    
    Combining the two results we get
    \begin{equation}
        \sin\angle(\vec x_R,{\vec u}_m) \leq \sqrt{ \rho_m^2+\frac{1}{1-\epsilon} \norm{Y_\perp \vec x_2}^2} \leq \sqrt{1+\frac{K_\epsilon \norm{A-\lambda I}^2}{\beta^2}} \rho_m,
    \end{equation} as required.
\end{proof}


\subsection{Sketched inverse iteration}
The correction step performs one step of inverse iteration in the sketched inner product. By \Cref{thm:ceequalsinviter}, the exact solution of \eqref{eq:rce} gives the expansion direction $\vec w_m := (A-{\theta}_mI)^{-1}{\vec u}_m$, so $\vec w_m \in \spn({V}_{m+1})$ and $\rho_{m+1} = \sin\angle(\vec x_R, {V}_{m+1}) \leq \sin\angle(\vec x_R,\vec w_m)$.

\begin{lemma}[Sketched inverse iteration accuracy] \label{lem:sketchedinviteraccuracy}
    Let $s_m = \sin\angle(\vec x_R, {\vec u}_m)$, $c_m = \cos\angle(\vec x_R, {\vec u}_m)$ and $\eta: = \sep(\lambda, A_{22})>0$. Assume $|\lambda - {\theta}_m| < \eta$ and $\kappa(\lambda)s_m<1$ with $\kappa(\lambda) = \frac{1}{|\vec x_L^*\vec x_R|}$. Set 
    \begin{equation}
        q_m:= \frac{\kappa(\lambda)|\lambda - {\theta}_m|}{\eta - |\lambda - {\theta}_m|} \cdot \frac{s_m}{c_m - s_m\sqrt{\kappa(\lambda)^2-1}}.
    \end{equation} If $q_m <1$, then the inverse iteration direction $\vec w_m = (A-{\theta_m}I)^{-1}{\vec u}_m$ satisfies
    \begin{equation}
        \sin\angle(\vec x_R,\vec w_m) \leq \frac{q_m}{1-q_m}.
    \end{equation}
\end{lemma}

\begin{proof}
    Let $\Pp = \frac{\vec x_R\vec x_L^*}{\vec x_L^*\vec x_R}$ be the spectral projector onto $\spn(\vec x_R)$. With $\norm{\vec x_L} = \norm{\vec x_R} = 1$, we have $\norm{\Pp} = \frac{1}{\abs{\vec x_L^*\vec x_R}} = \kappa(\lambda)$. By the projector identity \cite{Szyld06}, $\kappa(\lambda) = \norm{\Pp} = \norm{I-\Pp}$. 
    
	Without loss of generality we assume $\norm{\vec u_m}= 1$ as angles are invariant under scaling.
    Now decompose ${\vec u}_m$ as
    \begin{equation*}
        {\vec u}_m = \xi_m \vec x_R + \vec y_m, \quad \xi_m:= \frac{\vec x_L^*{\vec u}_m}{\vec x_L^*\vec x_R}, \quad \vec y_m:= (I-\Pp){\vec u}_m \in \spn(\vec x_L)^\perp.
    \end{equation*} We bound $\xi_m$ and $\norm{\vec y_m}$. Write ${\vec u}_m = c_m\vec x_R + s_m\vec v_m$ with $\vec v_m \perp \vec x_R$ and $\norm{\vec v_m} =1$, so that $\xi_m = c_m + s_m \frac{\vec x_L^*\vec v_m}{\vec x_L^*\vec x_R}$. Writing $\vec x_L = (\vec x_R^*\vec x_L)\vec x_R+(I-\vec x_R\vec x_R^*)\vec x_L$ and since $\vec x_R\perp \vec v_m$,
    \begin{equation*}
        \left\lvert\frac{\vec x_L^*\vec v_m}{\vec x_L^*\vec x_R}\right\rvert \leq \kappa(\lambda) \, \norm{(I-\vec x_R\vec x_R^*)\vec x_L} = \kappa(\lambda)\sqrt{1-\kappa(\lambda)^{-2}} = \sqrt{\kappa(\lambda)^2-1}.
    \end{equation*} Hence, we obtain
	\begin{equation}
		\label{eq:lemma-proof-xim-inequality}
		\abs{\xi_m}\geq c_m - s_m\sqrt{\kappa(\lambda)^2-1},
	\end{equation}
	and note that $\abs{\xi_m} >0$, since $\kappa(\lambda)s_m <1$ implies $s_m^2 (\kappa(\lambda)^2-1) < 1- s_m^2 = c_m^2$. For $\norm{\vec y_m}$, we note $\vec y_m = (I-\Pp)(c_m\vec x_R+s_m \vec v_m) = s_m\left(\vec v_m - \frac{\vec x_L^*\vec v_m}{\vec x_L^*\vec x_R}\vec x_R\right)$ using $\Pp \vec x_R = \vec x_R$ and $\Pp \vec v_m = \frac{\vec x_L^*\vec v_m}{\vec x_L^*\vec x_R}\vec x_R$, so $\vec v_m\perp \vec x_R$ gives
    \begin{equation}
		\label{eq:lemma-proof-ym-inequality}
        \norm{\vec y_m} = s_m \sqrt{1+\left\lvert \frac{\vec x_L^*\vec v_m}{\vec x_L^*\vec x_R}\right\rvert^2} \leq s_m \sqrt{1+(\kappa(\lambda)^2-1)}= s_m \kappa(\lambda).
    \end{equation}

    Now we compute the angle between $\vec x_R$ and $\vec w_m = (A-{\theta}_m I)^{-1}{\vec u}_m$. First, note that
    \begin{equation*}
        \vec w_m = (A-{\theta}_m I)^{-1}(\xi_m\vec x_R+\vec y_m) = \frac{\xi_m}{\lambda - {\theta}_m}\vec x_R + \vec z_m, \quad \vec z_m:=(A-{\theta}_m I)^{-1}\vec y_m.
    \end{equation*} Since $\vec y_m \in \spn(\vec x_L)^\perp$, which is $A$-invariant with $AX_\perp = X_\perp A_{22}$,
    \begin{equation*}
        \norm{(A-{\theta}_m I)^{-1}\vert_{\spn(\vec x_L)^\perp}} = (\sep({\theta}_m,A_{22}))^{-1} \leq (\eta - \abs{\lambda - {\theta}_m})^{-1}
    \end{equation*} by Weyl's inequality. Hence, using \cref{eq:lemma-proof-ym-inequality}, 
    \begin{equation}
		\label{eq:lemma-proof-zm-inequality}
        \norm{\vec z_m} \leq \frac{\norm{\vec y_m}}{\eta - \abs{\lambda - {\theta}_m}} \leq \frac{\kappa(\lambda)s_m}{\eta - \abs{\lambda - {\theta}_m}}.
    \end{equation} 
	Therefore
    \begin{align*}
        \sin\angle(\vec x_R,\vec w_m) \leq \tan\angle(\vec x_R,\vec w_m) = \frac{\norm{(I-\vec x_R\vec x_R^*)\vec w_m}}{\abs{\vec x_R^*\vec w_m}} &= \frac{\norm{(I-\vec x_R\vec x_R^*)\vec z_m}}{\left\lvert \frac{\xi_m}{\lambda - {\theta}_m} + \vec x_R^*\vec z_m \right\rvert} \\
        \leq \frac{\norm{\vec z_m}}{ \left\lvert \frac{\xi_m}{\lambda - {\theta}_m} \right\rvert- \norm{\vec z_m}}
    \end{align*} whenever $\left\lvert \frac{\xi_m}{\lambda - {\theta}_m} \right\rvert > \norm{\vec z_m}$, which is satisfied by the assumption $q_m <1$ as
    \begin{equation*}
        \norm{\vec z_m} \leq \frac{\kappa(\lambda)s_m}{\eta - \abs{\lambda - {\theta}_m}} < \frac{c_m - s_m\sqrt{\kappa(\lambda)^2-1}}{\abs{\lambda- {\theta}_m}} \leq \frac{\abs{\xi_m}}{\abs{\lambda- {\theta}_m}},
    \end{equation*} 
	where we used \cref{eq:lemma-proof-zm-inequality,eq:lemma-proof-xim-inequality}.
	Using the above we obtain
    \begin{equation*}
        \norm{\vec z_m} \left(\frac{\abs{\xi_m}}{\abs{\lambda- {\theta}_m}}\right)^{-1} \leq \frac{\kappa(\lambda)s_m}{\eta - |\lambda - {\theta}_m|} \cdot \frac{|\lambda - {\theta}_m|}{c_m - s_m \sqrt{\kappa(\lambda)^2-1}} = q_m
    \end{equation*}
    Finally, since $t\mapsto \frac{t}{1-t}$ is increasing on $[0,1)$, we get
    \begin{align*}
        \sin\angle(\vec x_R,\vec w_m) \leq \frac{\norm{\vec z_m}\left(\frac{\abs{\xi_m}}{\abs{\lambda- {\theta}_m}}\right)^{-1}}{ 1 - \norm{\vec z_m}\left(\frac{\abs{\xi_m}}{\abs{\lambda- {\theta}_m}}\right)^{-1}} \leq \frac{q_m}{1-q_m}.
    \end{align*}
\end{proof}

\begin{remark} \label{remark:inviternormal}
When $A$ is normal, $\vec x_L = \vec x_R$, so $\kappa(\lambda) = 1$ and the spectral projector $\Pp$ is orthogonal with $(A-{\theta}_m I)^{-1}\vec v_m \in \spn(\vec x_R)^\perp$. Therefore the proof simplifies leading to the bound
    \begin{equation}
        \sin\angle(\vec x_R,\vec w_m) \leq \frac{|\lambda - {\theta}_m|}{\eta - |\lambda - {\theta}_m|} \frac{s_m}{c_m}.
    \end{equation}
\end{remark}

\subsection{Local quadratic convergence}
Combining the lemmas proven in this section (\Cref{lem:sketchedRitzvalbnd,lem:sketchedRitzvecbnd,lem:sketchedinviteraccuracy}), we prove quadratic convergence of the randomized Jacobi-Davidson method (RJD). 
\begin{theorem}[Local quadratic convergence of RJD] \label{thm:convRJD}
    Let $(\lambda,\vec x_R,\vec x_L)$ be a simple eigentriplet of $A\in \C^{n\times n}$ with $\norm{\vec x_L} = \norm{\vec x_R} = 1$, $\kappa(\lambda) = 1/\abs{\vec x_L^*\vec x_R}$, separation $\eta = \sep(\lambda,A_{22}) > 0$ and adjoint residual norm $\nu(\lambda) = \norm{(A^*-\bar{\lambda}I)\vec x_R}$ as defined previously. Let $M_\lambda = \nu(\lambda) + (1+\epsilon)\norm{A - \lambda I}$, and let $C_{RR} = \sqrt{1+K_{\epsilon}\norm{A - \lambda I}^2/\beta^2}$ be the Ritz vector constant of \Cref{lem:sketchedRitzvecbnd}. Fix $\epsilon \in (0,1)$ and suppose
    \begin{enumerate}[label=(\roman*)]
        \item the correction equation \eqref{eq:rce} is solved exactly at each step;
        \item $S$ is an $\epsilon$-subspace embedding for $\spn([{V}_m,A{V}_m])$ at each step;
        \item the sketched uniform separation condition (\Cref{def:uniformseparation}) holds with constant $\beta$;
        \item the initial subspace angle $\rho_0 = \sin\angle(\vec x_R,{V}_0)$ satisfies $\rho_0 < \rho_*$ where 
        \begin{equation}
			\rho_*:= \frac{(1-\epsilon)\eta}{16\kappa(\lambda) M_\lambda C_{RR}^2}.
		\end{equation}
    \end{enumerate}
    Then, for every $m\geq 0$, the subspace angles $\rho_m = \sin\angle(\vec x_R,{V}_m)$ contract,
    \begin{equation} \label{eq:subspconvresult}
        \rho_{m+1} < \frac{\rho_m^2}{\rho_*}, 
    \end{equation} so $\rho_m$ converges at least quadratically; the Ritz vector angle $s_m = \sin\angle(\vec x_R,{u}_m)$ satisfies $s_m \leq C_{RR}\rho_m$, and it converges at the same rate. 
	\end{theorem}
\begin{proof}
    We proceed by induction on $m$. Assuming $\rho_m < \rho_*$, we show that the angle converges quadratically and that $\rho_{m+1} < \rho_*$. The base case assumption $\rho_0 < \rho_*$ ($m = 0$) is satisfied by the hypothesis $(iv)$.

    Assume $\rho_m < \rho_*$. By \Cref{lem:sketchedRitzvalbnd} and using $c_m$, $s_m \leq 1$, 
    \begin{equation*}
        \abs{\lambda - {\theta}_m} \leq \frac{s_m}{1 - \epsilon} ( c_m\nu(\lambda)+(s_m+\epsilon)\norm{A-\lambda I} ) \leq \frac{s_m}{1-\epsilon} M_\lambda = \frac{s_m}{1-\epsilon} M_\lambda.
    \end{equation*}
    Now we show that the assumptions of \Cref{lem:sketchedinviteraccuracy} are satisfied. First, note that 
	\begin{equation*}
		\rho_* = \frac{1}{4 \kappa(\lambda) C_{RR}} \cdot \frac{(1-\epsilon)\eta}{4 M_\lambda C_{RR}} \le \frac{1}{4 \kappa(\lambda) C_{RR}}
	\end{equation*}
	since $\eta \le \norm{A - \lambda I} \le M_\lambda$ and $C_{RR}\geq 1$, and hence by \cref{lem:sketchedRitzvecbnd} we have
	\begin{equation*}
		\kappa(\lambda)s_m\leq \kappa(\lambda)C_{RR}\rho_m \leq \kappa(\lambda)C_{RR}\rho_*\leq \frac{\kappa(\lambda)C_{RR}}{4\kappa(\lambda)C_{RR}} = \frac{1}{4}<1,
	\end{equation*}
	In particular $s_m \leq 1/4$, so $c_m = \sqrt{1-s_m^2} > 1/2$. Next, using the fact that
	\begin{equation*}
		\rho_* = \frac{(1-\epsilon) \eta}{2 M_\lambda C_{RR}} \cdot \frac{1}{8 \kappa(\lambda) C_{RR}} \le \frac{(1-\epsilon) \eta}{2 M_\lambda C_{RR}},
	\end{equation*}
	we obtain
    \begin{equation*}
        \abs{\lambda - {\theta}_m} \leq \frac{s_m}{1-\epsilon} M_\lambda \leq \frac{M_\lambda C_{RR}\rho_*}{1-\epsilon} \leq \frac{\eta}{2}<\eta,
    \end{equation*} and finally, using $c_m - s_m\sqrt{\kappa(\lambda)^2-1} > c_m - \kappa(\lambda)s_m \geq \frac{1}{2} - \frac{1}{4} = \frac{1}{4}$, we have	
    \begin{align*}
        q_m &:= \frac{\kappa(\lambda)s_m}{\eta - |\lambda - {\theta}_m|} \cdot \frac{|\lambda - {\theta}_m|}{c_m - s_m \sqrt{\kappa(\lambda)^2-1}} \leq \frac{\kappa(\lambda)s_m}{\eta/2}\cdot\frac{M_\lambda s_m/(1-\epsilon)}{1/4}\\ 
		&< \frac{8\kappa(\lambda)M_\lambda}{(1-\epsilon)\eta} s_m^2 \leq \frac{1}{2C_{RR}^2 \rho_*} s_m^2 \leq \frac{1}{2 \rho_*} \rho_m^2.
    \end{align*}
	Now $q_m < \frac{1}{2 \rho_*} \rho_m^2 < \frac{1}{2\rho_*} \rho_*\rho_m \leq \frac{1}{2} < 1$ so by
    \Cref{lem:sketchedinviteraccuracy} we obtain $\rho_{m+1} \leq \sin\angle(\vec x_R,\vec w_m) \leq \frac{q_m}{1-q_m}$ where $\vec w_m = (A-{\theta}_mI)^{-1}{\vec u}_m \in \spn({V}_{m+1})$.
    Therefore,
    \begin{equation}
        \rho_{m+1} \leq \frac{q_m}{1-q_m} \leq 2q_m < \frac{\rho_m^2}{\rho_*}
    \end{equation} 
	and $\rho_{m+1} < \rho_m^2 / \rho_* < \rho_*\rho_m / \rho_* = \rho_m < \rho_*$, which completes the induction.
\end{proof}

\Cref{thm:convRJD} bounds the subspace angle $\rho_m = \sin\angle(\vec x_R,{V}_{m})$, the geometric quantity tracked by the analysis. In practice one observes the Ritz vector angle $s_m = \sin\angle(\vec x_R,{\vec u}_m)$ instead, since ${\vec u}_m$ is the computed approximation to $\vec x_R$ while ${V}_{m}$ is implicit. The next result phrases the convergence guarantee in terms of $s_m$.

\begin{corollary}[Local quadratic convergence in the Ritz vector] \label{cor:convRitzvec}
    Under the hypotheses of \Cref{thm:convRJD}, the Ritz vector angle $s_m = \sin\angle(\vec x_R, {\vec u}_m)$ converges to zero quadratically by
    \begin{equation}
        s_{m+1} \leq L s_m^2, \quad L = \frac{1}{C_{RR} \rho_*}.
    \end{equation}
\end{corollary}
\begin{proof}
    From the proof of \Cref{thm:convRJD}, $\rho_{m+1} \leq 2q_m \leq \frac{1}{C_{RR}^2 \rho_*}s_m^2$. Therefore, by \Cref{lem:sketchedRitzvecbnd}
    \begin{equation}
        s_{m+1} \leq C_{RR}\rho_{m+1} \leq C_{RR} \frac{1}{C_{RR}^2 \rho_*} s_m^2 =  Ls_m^2.
    \end{equation}
\end{proof}

We close this section with an observation on the normal case. This is not the setting at which the present work is aimed, but it makes clear what effect sketching has when the problem does have structure.

\begin{remark}[Normal and Hermitian $A$] \label{remark:normalcase}
    Let $A$ be normal. Then the left and right eigenvectors coincide, so $\kappa(\lambda) = 1$ and $\nu(\lambda) = 0$, and the separation reduces to the spectral gap $\eta = \gap_\lambda:=\min\limits_{\sigma \in \lambda(A), \sigma \neq \lambda}|\sigma - \lambda|$. Following the proof of \Cref{thm:convRJD} with these values, and using \Cref{remark:inviternormal} in place of \Cref{lem:sketchedinviteraccuracy}, the contraction becomes
    \begin{equation*}
        \rho_{m+1} \leq \frac{4\norm{A - \lambda I}C_{RR}^2}{(1-\epsilon)\gap_{\lambda}}(\epsilon \rho_m^2 + C_{RR}\rho_m^3), \quad s_{m+1} \leq \frac{4\norm{A- \lambda I}C_{RR}}{(1-\epsilon)\gap_{\lambda}} (\epsilon s_m^2 + s_m^3),
    \end{equation*} for sufficiently small $\rho_m$ (and $s_m$), which is quadratic for a fixed $\epsilon$. The quadratic term comes from the $\epsilon$ contribution to the sketched Rayleigh quotient in \Cref{lem:sketchedRitzvalbnd}. It disappears if the shift used in the residual and in the correction equation~\eqref{eq:rce} is the standard Rayleigh quotient $\widetilde{\theta}_m$ of $\vec u_m$, which satisfies $|\lambda - \widetilde{\theta}_m| \leq \norm{A - \lambda I} s_m^2$ when $\nu(\lambda) = 0$. \Cref{thm:ceequalsinviter} allows an arbitrary shift and \Cref{lem:sketchedRitzvecbnd} does not involve the shift, so the same argument gives local cubic convergence,
    \begin{equation*}
        \rho_{m+1} \leq \frac{4\norm{A-\lambda I}C_{RR}^3}{\gap_\lambda}\rho_m^3, \quad s_{m+1} \leq \frac{4\norm{A-\lambda I}C_{RR}}{\gap_\lambda} s_m^3,
    \end{equation*} at a cost of two extra inner products of length $n$ per iteration, since $A\vec u_m = W\vec y_m$ is already available. Note that with this shift $\vec r_m$ is no longer $S$-orthogonal to $\vec u_m$, which is the situation for which the projected right-hand side of \eqref{eq:rce} is needed.

    We nevertheless do not recommend \Cref{alg:rjd} for Hermitian $A$. Sketching destroys the symmetry of the extraction, because $H = (SV)^*(SAV)$ is not Hermitian even when $A$ is. The projected problem is then a non-Hermitian eigenvalue problem, which is several times more expensive than the Hermitian one of the same size, so the saving in the orthogonalization is outweighed. One could instead keep the sketch-orthonormal basis and extract from the Hermitian definite pencil $V^*AV \vec y = \theta V^*V\vec y$, which restores the structure of the projected problem; but maintaining $V^*AV$ and $V^*V$ reintroduces the length-$n$ inner products that sketching was introduced to avoid. For Hermitian problems this second option is preferable, and it is implemented for CPUs and GPUs in~\cite{Gubleretal26}, where the saving in the orthogonalization outweighs the cost of maintaining $V^*AV$ and $V^*V$. Our interest is thus in the non-Hermitian case, where no such structure is available and where, by \Cref{thm:convRJD}, sketching costs nothing in the order of local convergence.
\end{remark}


\section{Harmonic and refined extraction} \label{sec:harmonicrefined}
Standard sketched extraction (\Cref{alg:rjd}) is reliable for exterior eigenvalues, where the uniform separation condition (\Cref{def:uniformseparation}) holds. For an interior target, it can fail both to select the correct value and to converge to the corresponding Ritz vector. Harmonic extraction (\Cref{subsec:harmonic}) addresses the value: it is Rayleigh-Ritz for a shift-invert operator in which the target is dominant, relocating the governing separation condition to a setting that is typically far more favorable for interior eigenvalues. Refined extraction (\Cref{subsec:refined}) addresses the vector, and removes the uniform separation condition. Combining the two (\Cref{subsec:harmonic+refined}) gives quadratic convergence for interior eigenvalues, with all the analysis inherited from \Cref{sec:analysis}. Throughout, $V_m$ is the sketch-orthonormal search basis, $\rho_m = \sin\angle(\vec x_R,V_m)$, and $(\lambda,\vec x_R, \vec x_L), \kappa(\lambda),A_{22}, \eta = \sep(\lambda, A_{22}),X_\perp, M_\lambda, K_\epsilon$ are as in \Cref{sec:analysis}.

\subsection{Harmonic extraction} \label{subsec:harmonic}
Given a target $\tau \in \C$, harmonic Rayleigh-Ritz \cite{Morgan91,MorganZeng98} enforces $(A- \theta I)\vec u \perp (A-\tau I)V_m$ for $\vec u = V_m \vec y$ \cite[Thm.~2.3]{MorganZeng98}. Sketching the inner product and writing $A - \theta I = (A- \tau I) - (\theta-\tau )I$ gives the $m\times m$ generalized eigenvalue problem
\begin{equation} \label{eq:sketchedharmonic}
    (S(A - \tau I)V_m)^*(S(A-\tau I)V_m) \vec y = (\theta - \tau)(S(A-\tau I)V_m)^*(SV_m) \vec y.
\end{equation}
The pencil costs $\order(sm^2)$ to assemble and $\order(m^3)$ to solve with no new products with $A$; as $\spn((A-\tau I)V_m) \subseteq \spn([V_m,AV_m])$, the sketch $S$ of \Cref{alg:rjd} already embeds it. Moreover $S(A-\tau I)V_m = B-\tau Q$ in the notation of \Cref{alg:rjd}, so the pencil is assembled from quantities that are already stored and no new sketch is required. Since $(A-\tau I)^{-1}$ sends the eigenvalues nearest $\tau$ to the exterior of its spectrum, the wanted interior eigenvalue becomes dominant, which is where the extraction is reliable \cite{MorganZeng98}. The extraction, however, still depends on a separation condition, now applied to the transformed pencil.

\begin{lemma} \label{lem:sketchedharmonic}
    Let $S$ be an $\epsilon$-subspace embedding for $\spn((A-\tau I) V_m)$, and let $\theta \neq \tau$. Then $(\theta, \vec u = V_m \vec y)$ solves \eqref{eq:sketchedharmonic} if and only if $(\mu, \vec y)$ with $\mu = (\theta - \tau)^{-1}$ is a sketched Rayleigh-Ritz eigenpair of $(A-\tau I)^{-1}$ on $\spn((A-\tau I)V_m)$, with Ritz vector $\widehat{\vec u} = (A-\tau I)V_m\vec y$. 
\end{lemma}
\begin{proof}
    The sketched Rayleigh-Ritz condition for $(A-\tau I)^{-1}$ on $\spn((A-\tau I)V_m)$ is $(S(A-\tau I)V_m)^*S((A-\tau I)^{-1} - \mu I)(A-\tau I)V_m \vec y = 0$, which is \eqref{eq:sketchedharmonic} with $\mu = (\theta - \tau)^{-1}$.
\end{proof}

By \Cref{lem:sketchedharmonic}, the bounds of \Cref{sec:analysis} apply to $(A-\tau I)^{-1}$ on $\spn((A-\tau I)V_m)$, whose target eigenvalue is $\mu_* = (\lambda - \tau)^{-1}$ with the same right eigenvector $\vec x_R$.

\begin{theorem}
    Let $S$ be an $\epsilon$-subspace embedding for $\spn([V_m,AV_m])$, let $\tau \notin \lambda(A)$, and let $\theta_m$ be the harmonic Ritz value from \eqref{eq:sketchedharmonic} targeting $\lambda$, with $\mu_m = (\theta_m-\tau)^{-1}$ and $\norm{\widehat{\vec u}_m}_S=1$. Assume that the uniform separation condition holds with constant $\beta_R$ for the Rayleigh--Ritz projection of $(A-\tau I)^{-1}$ onto $\spn((A-\tau I)V_m)$, represented in an $S$-orthonormal basis, and that $\abs{\mu_m - \mu_*} \leq |\mu_*|/2$. Then
    \begin{equation} \label{eq:harmonicvaluebnd}
        |\theta_m - \lambda| \leq L_\theta \rho_m, \quad L_\theta = \frac{2\kappa_2(A-\tau I)C_RM_\lambda}{1-\epsilon},
    \end{equation} where $C_R = \sqrt{1+K_\epsilon \norm{(A-\tau I)^{-1} - \mu_* I}^2/\beta_R^2}$ and $M_\lambda = \nu(\lambda)+(1+\epsilon) \norm{A-\lambda I}$.
\end{theorem}
\begin{proof}
    Since $\spn([(A-\tau I)V_m,(A-\tau I)^{-1}((A-\tau I)V_m)]) \subseteq \spn([V_m,AV_m])$, the sketch $S$ embeds the space required by \Cref{lem:sketchedRitzvalbnd,lem:sketchedRitzvecbnd} applied to $(A-\tau I)^{-1}$ on $\spn((A-\tau I)V_m)$. From $(A-\tau I)^{-1} - \mu_*I = -\mu_*(A-\tau I)^{-1}(A-\lambda I)$ we get $\norm{(A-\tau I)^{-1} - \mu_* I} \leq |\mu_*|\norm{(A-\tau I)^{-1}}\norm{A-\lambda I}$ and, using that $(A^*-\bar{\lambda}I)$ and $(A^*-\bar{\tau}I)^{-1}$ commute, $\nu_{(A-\tau I)^{-1}}(\mu_*) = |\mu_*|\norm{(A^*-\bar{\tau}I)^{-1}(A^*-\bar{\lambda}I) \vec x_R} \leq |\mu_*|\norm{(A-\tau I)^{-1}}\nu(\lambda)$; hence
    \begin{equation*}
        \nu_{(A-\tau I)^{-1}}(\mu_*)+(1+\epsilon)\norm{(A-\tau I)^{-1} - \mu_*I} \leq |\mu_*|\norm{(A-\tau I)^{-1}}M_{\lambda}.
    \end{equation*}
    Next, if $\vec v \in \spn(V_m)$ attains $\norm{\vec x_R - \vec v} = \rho_m$, then $\mu_*(A-\tau I)\vec v \in \spn((A-\tau I)V_m)$ and $\norm{\vec x_R - \mu_*(A-\tau I)\vec v}= |\mu_*|\norm{(A-\tau I)(\vec x_R - \vec v)}$, so $\sin\angle(\vec x_R, (A-\tau I)V_m) \leq |\mu_*|\norm{A-\tau I}\rho_m$ and \Cref{lem:sketchedRitzvecbnd} gives $\sin\angle(\vec x_R, \widehat{\vec u}_m)\leq C_R \abs{\mu_*}\norm{A-\tau I}\rho_m$. Combining these in \Cref{lem:sketchedRitzvalbnd}, and using $c,s\leq 1$ there,
    \begin{equation*}
        \abs{\mu_m-\mu_*} \leq \frac{\abs{\mu_*}^2 \kappa_2(A-\tau I)C_R M_{\lambda}}{1-\epsilon}\rho_m.
    \end{equation*} Finally $\abs{\theta_m-\lambda} = \abs{\mu_m - \mu_*}/(\abs{\mu_m}\abs{\mu_*})$, and the hypothesis $\abs{\mu_m-\mu_*}\leq |\mu_*|/2$ gives $\abs{\mu_m} \geq \abs{\mu_*}/2$, so $\abs{\theta_m - \lambda} \leq 2\frac{\abs{\mu_m-\mu_*}}{\abs{\mu_*}^2}$, which completes the proof.
\end{proof}

The hypothesis $\abs{\mu_m - \mu_*}\leq \abs{\mu_*}/2$ is a locality assumption of the same nature as $\rho_0 \leq \rho_*$ in \Cref{thm:convRJD}, and holds as soon as $\rho_m \leq (1-\epsilon)/(2\abs{\mu_*}\kappa_2(A-\tau I)C_RM_{\lambda})$. The harmonic value is thus as accurate as the standard Rayleigh quotient, up to constants reflecting the shift-invert transformation, the factor $\kappa_2(A-\tau I)$ being the price of that transformation. The harmonic vector, however, requires an additional separation requirement that makes extraction more fragile \cite{Wu2017}. We therefore obtain the vector by refinement instead.

\subsection{Refined extraction} \label{subsec:refined}
The refined Ritz vector \cite{Jia97, JiaStewart01} is the unit vector in the search subspace with smallest residual at a target $\theta_m$; its sketched analogue replaces the residual by its sketched norm, 
\begin{equation} \label{eq:sketchedrefined}
    \vec u_m^{\mathrm{ref}} = V_m \vec c_m, \quad \vec c_m = \argmin_{\norm{\vec c} = 1} \norm{S(A-\theta_mI)V_m \vec c},
\end{equation} so $\vec c_m$ is the trailing right singular vector of $S(A-\theta_m I)V_m$, computed in $\order(sm^2)$ by a thin SVD; as in \eqref{eq:sketchedharmonic}, $S(A-\theta_m I)V_m = B-\theta_m Q$ is available from stored quantities. Approximating the right singular vectors through a sketch is studied in \cite{gilbert2013sketched,pn23nullspace}.

\begin{lemma} \label{lem:sketchedrefinedconv}
Let $(\lambda, \vec x_R)$ be simple, and suppose $\abs{\lambda - \theta_m} < \sigma_{\min}(A_{22} - \theta_m I)$ where $A_{22}$ is the deflated block of $A$ at $\lambda$ as in \Cref{sec:analysis}. If $V_m$ is sketch-orthonormal and $S$ is an $\epsilon$-subspace embedding for $\spn([V_m,AV_m])$, then 
    \begin{equation} \label{eq:sketchedrefinedbnd}
        \sin\angle(\vec x_R, \vec u_m^{\mathrm{ref}}) \leq \frac{1+\epsilon}{\sqrt{1-\epsilon}} \kappa_2([\vec x_R, X_\perp])\frac{\min\limits_{\norm{\vec c} = 1}\norm{(A-\theta_m I )V_m \vec c}}{\sigma_{\min}(A_{22} - \theta_m I)},
    \end{equation} so $\vec u_m^{\mathrm{ref}}\rightarrow \vec x_R$ whenever the minimal residual does, with no uniform separation condition.
\end{lemma}
\begin{proof}
     Applying \cite[Thm~6.1]{JiaStewart01} to the unit vector $V_m \vec c_m / \norm{V_m \vec c_m}$, we obtain
    \begin{equation*}
         \sin\angle(\vec x_R, \vec u_m^{\mathrm{ref}}) \leq  \kappa_2([\vec x_R, X_\perp])\frac{\norm{(A-\theta_m I )V_m \vec c_m}}{\norm{V_m \vec c_m}\sigma_{\min}(A_{22} - \theta_m I)}.
    \end{equation*} Since $V_m$ is sketch-orthonormal, $\norm{V_m \vec c_m} \geq (1+\epsilon)^{-1/2}$. Moreover, by \cite[Remark~2.1]{pn23nullspace}, \begin{equation*}
        \norm{(A - \theta_m I)V_m \vec c_m} \leq \sqrt{\frac{1+\epsilon}{1-\epsilon}}\min\limits_{\norm{\vec c} = 1}\norm{(A-\theta_m I )V_m \vec c}.
    \end{equation*} Substituting these two bounds into the inequality yields the desired result.
\end{proof}

The right-hand side of \eqref{eq:sketchedrefinedbnd} is controlled by the subspace angle. Let $\vec v = V_m\vec d \in \spn(V_m)$ attain $\norm{\vec x_R - \vec v} = \rho_m$; taking $\vec c = \vec d/\norm{\vec d}$ in the minimum, and using $\norm{\vec d} \geq \sqrt{1-\epsilon}\norm{\vec v} = \sqrt{1-\epsilon}\sqrt{1-\rho_m^2}$ together with $(A-\theta_m I)\vec x_R = (\lambda - \theta_m)\vec x_R$, gives
\begin{equation} \label{eq:refinedminbnd}
    \min\limits_{\norm{\vec c} = 1}\norm{(A-\theta_m I )V_m \vec c} \leq \frac{(\norm{A-\lambda I} + \abs{\lambda - \theta_m})\rho_m + \abs{\lambda - \theta_m}}{\sqrt{1-\epsilon}\sqrt{1-\rho_m^2}}.
\end{equation}

\subsection{Refined harmonic extraction} \label{subsec:harmonic+refined}
For interior eigenvalues of strongly non-normal $A$, take $\theta_m$ to be the harmonic Ritz value from \eqref{eq:sketchedharmonic} and compute the refined vector $\vec u_m^{\mathrm{ref}}$ from \eqref{eq:sketchedrefined} at that shift, the sketched analogue of Jia's refined harmonic Rayleigh-Ritz \cite{Jia05}. The harmonic value keeps $\abs{\lambda - \theta_m} \leq L_\theta \rho_m$ accurate, and the refined vector is accurate with no uniform separation condition. Substituting $\abs{\lambda - \theta_m} \leq L_\theta \rho_m$ into \eqref{eq:refinedminbnd} and then into \eqref{eq:sketchedrefinedbnd}, and using $\sigma_{\min}(A_{22} - \theta_m I) \geq \eta - |\lambda - \theta_m|$ and $\kappa_2([\vec x_R, X_{\perp}]) \leq 2\kappa(\lambda)$, gives
\begin{equation*}
    \sin\angle(\vec x_R, \vec u_m^{\mathrm{ref}}) \leq C_{RH}\rho_m, \quad C_{RH} = \frac{5(1+\epsilon)\kappa(\lambda)(\norm{A-\lambda I}+2L_\theta)}{(1-\epsilon)\eta},
\end{equation*} whenever $\rho_m \leq 1/2$ and $\abs{\lambda - \theta_m}\leq \eta/2$, so that $\sqrt{1-\rho_m^2} \geq \sqrt{3}/2$ and $\sigma_{\min}(A_{22} - \theta_m I)\geq \eta/2$. Since the exact correction step is sketched inverse iteration (\Cref{thm:ceequalsinviter}), feeding this bound and \eqref{eq:harmonicvaluebnd} into the contraction of \Cref{lem:sketchedinviteraccuracy} makes the induction of \Cref{thm:convRJD} apply similarly, giving local quadratic convergence
\begin{equation*}
    \sin\angle(\vec x_R, V_{m+1}) \leq \frac{16\kappa(\lambda)L_\theta C_{RH}}{\eta} \sin^2\!\!\angle(\vec x_R, V_m),
\end{equation*} for interior eigenvalues. A separation condition is required only for the value through $L_\theta$, not for the vector. When $\tau$ is very close to $\lambda$, the factor $\kappa_2(A-\tau I)$ in $L_\theta$ becomes large and the harmonic value loses its guarantee. In that regime the Rayleigh quotient of the refined vector, $(\vec{u}_m^{\mathrm{ref}})^* A \vec{u}_m^{\mathrm{ref}}/\norm{\vec{u}_m^{\mathrm{ref}}}^2$, can be used instead as it converges to $\lambda$ once $\vec u_m^{\mathrm{ref}} \rightarrow \vec x_R$, regardless of $\abs{\lambda - \tau}$. This costs two extra inner products of length $n$ per iteration since $A\vec u_m^{\mathrm{ref}} = W\vec c_m$ is already available.

\section{Numerical Experiments} \label{sec:experiments}
We evaluate RJD (\Cref{alg:rjd}) against classical JD (\Cref{alg:jd}) on non-Hermitian eigenvalue problems drawn from the SuiteSparse Matrix Collection~\cite{SuiteSparse} together with synthetic matrices of prescribed structure. We report three experiments: order of convergence and the effect of inexact correction solves (\Cref{subsec:inexactce}), multiple eigenvalues on real large-scale problems (\Cref{subsec:multeig}), and extraction strategies for interior eigenvalues (\Cref{subsec:interioreigtest}). All experiments use the \texttt{sparsestack} sparse embedding sketching operator with $\zeta = 4$ nonzeros per column and sketch size $s = 4m_{\max}$. All experiments were performed in MATLAB version 2024b using double precision arithmetic.

\subsection{Inexact solves for correction equation and order of convergence} \label{subsec:inexactce}
\Cref{thm:convRJD} predicts local quadratic convergence for both methods when the correction equation is solved exactly. We confirm this on a controlled synthetic matrix $A = I+ G/\sqrt{n}$ with $n = 300$ and $G$ an $n\times n$ matrix with i.i.d. standard Gaussian entries, targeting the eigenvalue of largest real part with tolerance $10^{-10}$, restart dimensions $m_{\min} = 20$, $m_{\max} = 40$, and the correction equation solved exactly. The starting vector is a random Gaussian vector. \Cref{fig:ordconv} shows the resulting trajectory and fitted local order, $2.26$ for JD and $2.34$ for RJD, in agreement with \Cref{thm:convRJD}.

\begin{figure}[t]
  \centering
  \includegraphics[width=0.48\linewidth]{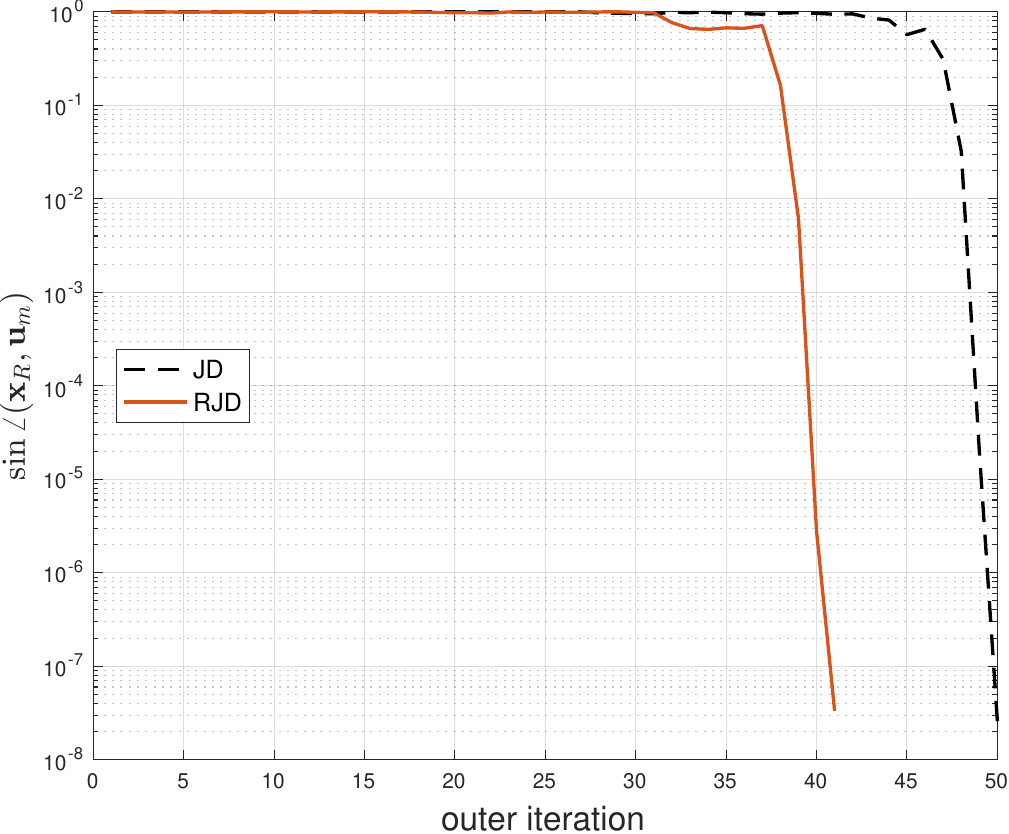}
  \hfill
  \includegraphics[width=0.48\linewidth]{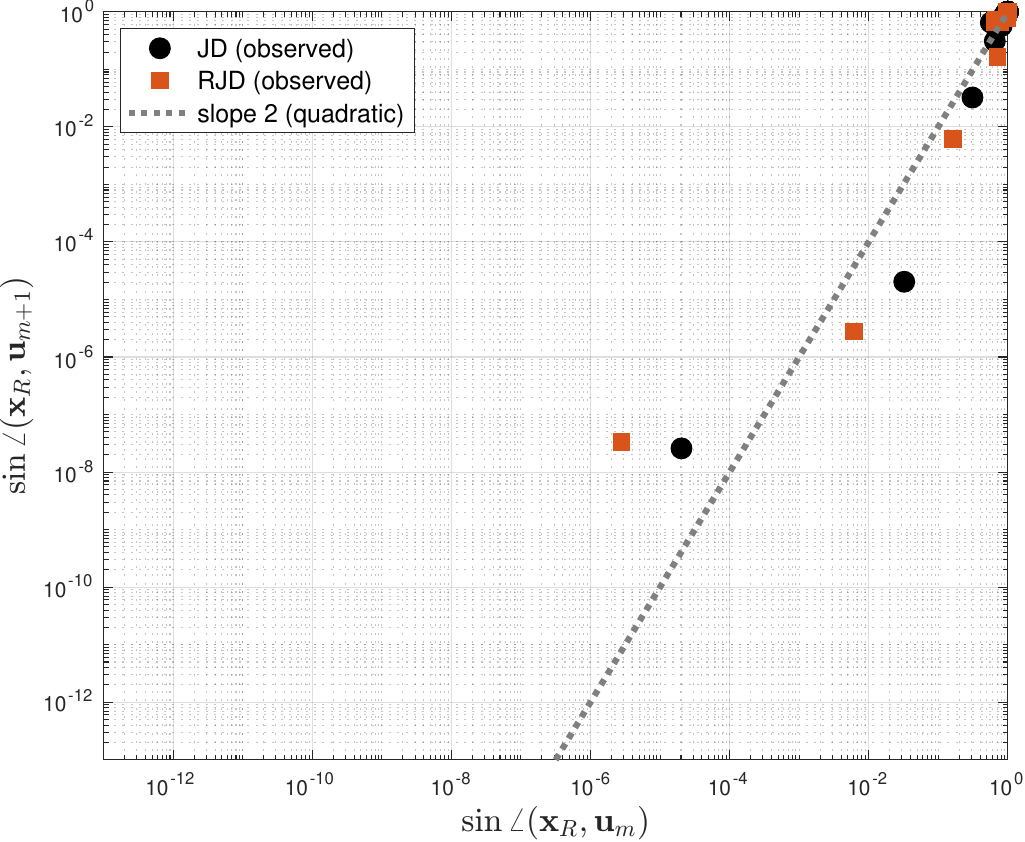}
  \caption{Eigenvector convergence on a synthetic problem ($n = 300$, exact correction equation solve). Left: $\sin\angle(\vec x_R, \vec u_m)$ against outer iteration, JD (dashed) and RJD (solid). Right: $\sin\angle(\vec x_R, \vec u_{m+1})$ against $\sin\angle(\vec x_R, \vec u_m)$ for both methods against the theoretical slope-$2$ reference line. The fitted local order is $2.26$ for JD and $2.34$ for RJD, consistent with \Cref{thm:convRJD}.}
  \label{fig:ordconv}
\end{figure}

We repeat the experiment on the real, non-symmetric matrix $\mathtt{garon2}$ from the SuiteSparse collection ($n = 13535$ with $373235$ nonzero entries), targeting the dominant eigenvalue $\lambda \approx 8.5451$ with tolerance $10^{-10}$ and restart dimensions $m_{\min} = 20$, $m_{\max} = 40$. The starting vector is warm-started by one step of power iteration on a Gaussian random vector, followed by normalization. The correction equation is solved either exactly or by GMRES with a small, fixed inner iteration budget of $5, 10$, or $20$ iterations. \Cref{fig:inexactsolves} shows that JD and RJD track each other closely at every level of correction accuracy, with one or the other ahead at different stages, but neither consistently better, and \Cref{table:ordconv} reports the corresponding fitted local order. At all three fixed budgets the fitted order is close to $1$ for both methods with an increase in order when the correction equation is solved more accurately; the exact solve produces a superlinear rate. This is consistent with the classical inexact Jacobi-Davidson analysis~\cite{BaiMiao17}: the convergence order degrades from quadratic toward linear as the correction solve is made less accurate. What matters here is that the degradation is a property of the correction accuracy and is shared by both methods, so it is not something sketching introduces. RJD following the curve of JD indicates that sketching does not cost local convergence order for nonsymmetric eigenvalue problems.

\begin{figure}[t]
  \centering
  \includegraphics[width=0.48\linewidth]{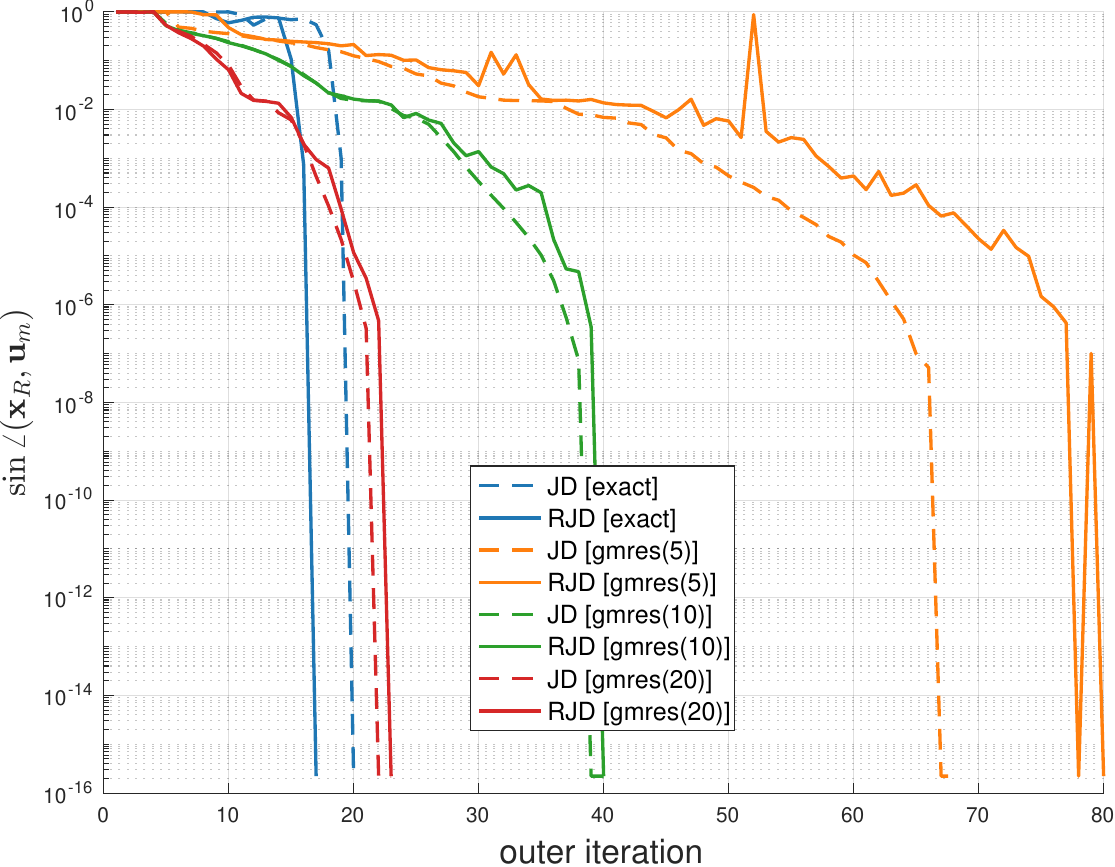}
  \hfill
  \includegraphics[width=0.48\linewidth]{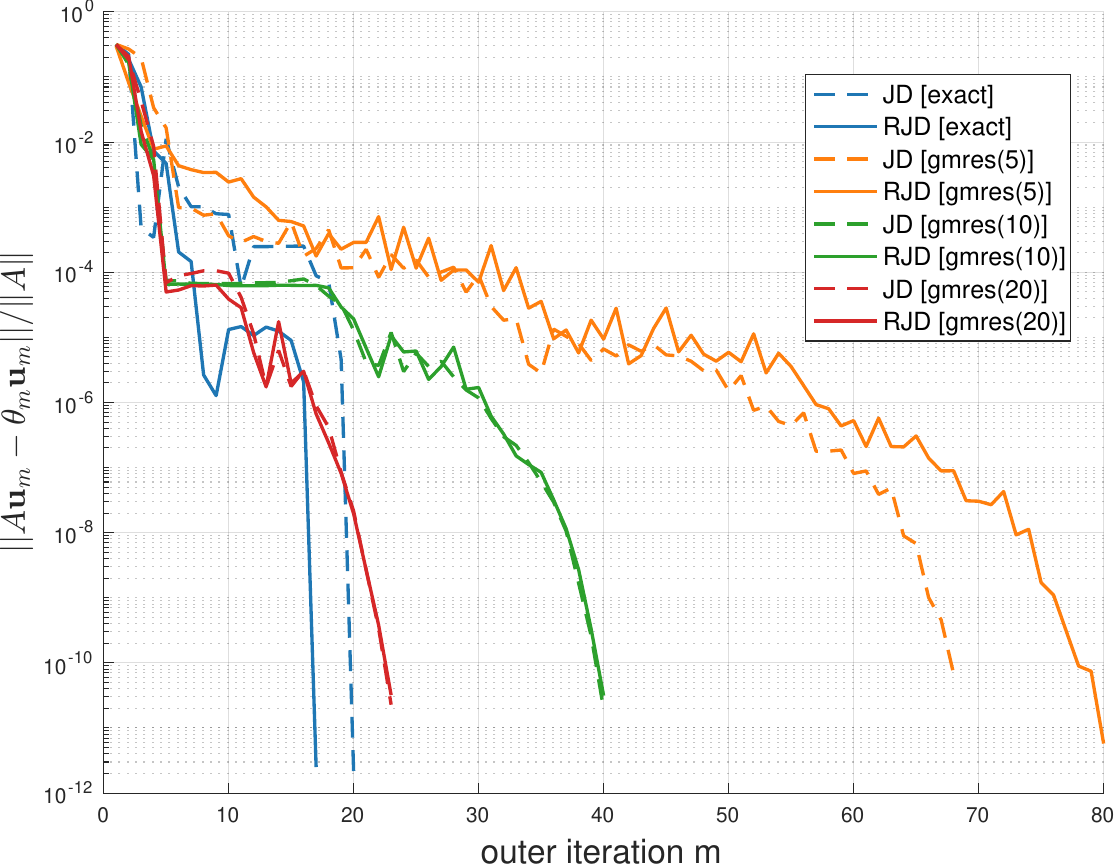}
  \caption{Convergence on the real, non-symmetric matrix $\mathtt{garon2}$ $(n = 13535)$, targeting the dominant eigenvalue $\lambda \approx 8.5451$, JD (dashed) and RJD (solid), for the correction equation solved exactly and by GMRES with a fixed inner budget of $5, 10,$ and $20$ iterations. Left: $\sin\angle(\vec x_R,\vec u_m)$ against outer iteration. Right: Residual norm against outer iteration.}
  \label{fig:inexactsolves}
\end{figure}

\begin{table}[ht] 
	\caption{Observed local order of convergence on $\mathtt{garon2}$ as a function of correction equation accuracy (fixed GMRES inner iteration budget or exact solve).}
	\label{table:ordconv}
\centering
\begin{tabular}{lcccc}
\hline
 & GMRES(5) & GMRES(10) & GMRES(20) & Exact \\
\hline
JD  & 1.05 & 1.10 & 1.16 & 3.71 \\
RJD & 1.06 & 1.11 & 1.13 & 3.20 \\
\hline
\end{tabular}
\end{table}

\subsection{Multiple eigenvalues} \label{subsec:multeig}
In this experiment, we show that JD and RJD reach the same eigenpairs at comparable cost on large, nonsymmetric problems.\footnote{As the nonsymmetric eigenvalue problem is highly sensitive, for difficult problems JD and/or RJD may converge to the wrong eigenpair or not converge at all.}

The result is shown in \Cref{fig:multeig}. We consider the three test matrices \texttt{Goodwin\_040} ($n=17922$ with $561677$ nonzeros), \texttt{Goodwin\_071} ($n=56021$ with $1797934$ nonzeros) and  \texttt{M80PI\_n} ($n=4182$ with $10261$ nonzeros). For all three matrices, we target $k=6$ eigenvalues\footnote{We use hard locking~\cite[Ch.~4.3.4]{AlgEigProbBook} to lock the converged vectors, so we (sketch-)orthogonalize the basis against the locked vectors.} with a random Gaussian starting vector, the correction equation is solved with $5$ iterations of GMRES, restart dimensions are $m_{\min}=60$, $m_{\max}=120$, and the tolerance is $10^{-8}$.
For the matrix \texttt{Goodwin\_040} we target six interior eigenvalues near $\tau=10^{-3}(1.6+2.4i)$, a small-magnitude complex target well away from the extremal part of the spectrum, for \texttt{Goodwin\_071} we target six eigenvalues with the smallest magnitude, and for \texttt{M80PI\_n} we target the six eigenvalues with largest imaginary part. All experiments identify and converge to the correct target eigenvalues. The two methods require a comparable number of outer iterations on all three problems, with RJD requiring slightly more; since the number of matrix-vector products and the number of correction solves are the same per outer iteration, the costs of the two methods are comparable in those units. However, RJD has lower orthogonalization cost, which may offset a slightly higher iteration count.

\begin{figure}[!t]
    \centering
    \begin{minipage}[c]{0.48\linewidth}
        \centering
        \includegraphics[height=0.82\linewidth]{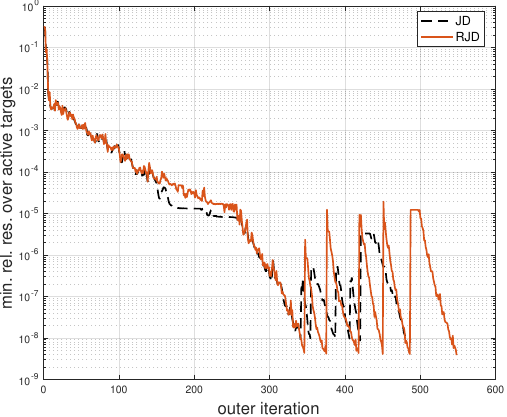}
    \end{minipage}
    \hfill
    \begin{minipage}[c]{0.48\linewidth}
        \centering
        \includegraphics[height=0.82\linewidth]{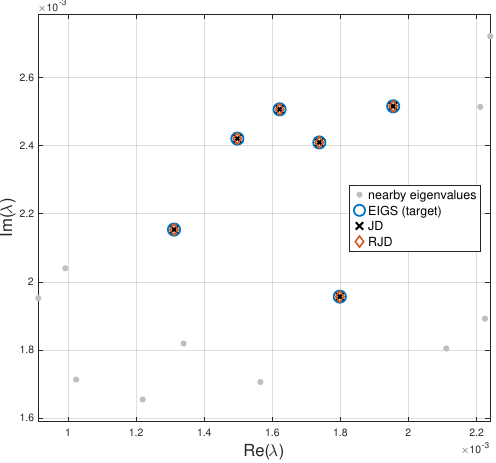}
    \end{minipage}


    \begin{minipage}[c]{0.48\linewidth}
        \centering
        \includegraphics[height=0.82\linewidth]{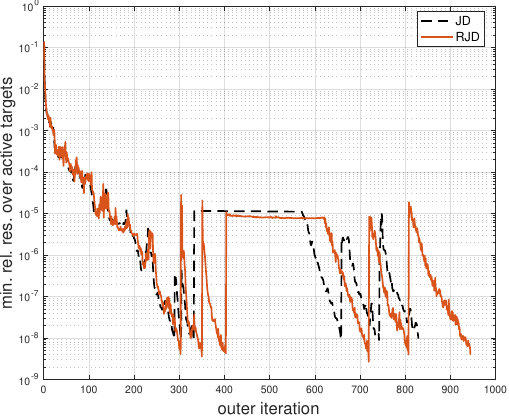}
    \end{minipage}
    \hfill
    \begin{minipage}[c]{0.48\linewidth}
        \centering
        \includegraphics[height=0.82\linewidth]{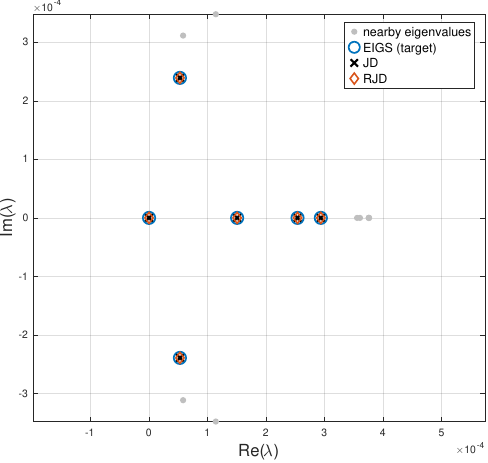}
    \end{minipage}

    \begin{minipage}[c]{0.48\linewidth}
        \centering
        \includegraphics[height=0.82\linewidth]{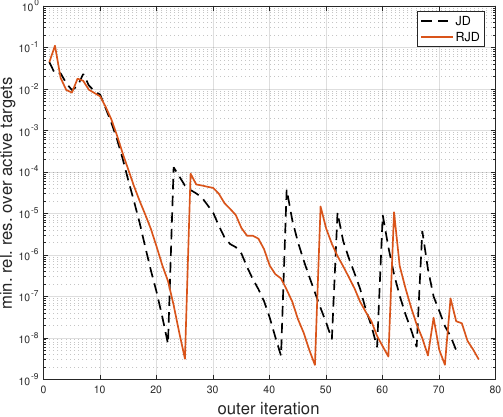}
    \end{minipage}
    \hfill
    \begin{minipage}[c]{0.48\linewidth}
        \centering
        \includegraphics[height=0.82\linewidth]{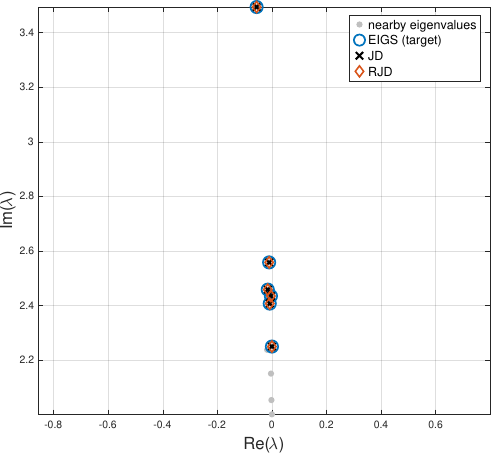}
    \end{minipage}

    \caption{Multiple-eigenvalue experiments on three real matrices. Top: $\mathtt{Goodwin\_040}$. Middle: $\mathtt{Goodwin\_071}$. Bottom: $\mathtt{M80PI\_n}$. Left column: residual against outer iteration; right column: computed eigenvalues in the complex plane against a reference computed by $\mathtt{eigs}$.}
    \label{fig:multeig}
\end{figure}

\subsection{Extraction strategies for interior eigenvalues} \label{subsec:interioreigtest}

\Cref{sec:harmonicrefined} extends harmonic and refined extraction, developed for interior targets beyond the exterior-target setting of \Cref{thm:convRJD}, to \Cref{alg:rjd}. Here we test whether that extension preserves JD's reliability, for each extraction strategy, rather than which strategy is best.

The test matrix is $A = X\diag(\vec\lambda)X^{-1}$ with $n=50$, $\kappa_2(X)=2000$, and $\vec\lambda$ consisting of $20$ eigenvalues evenly spaced in $[-15,15]$ and $30$ densely clustered in $[-0.5,0.5]$. The target is the $15$th of the clustered eigenvalues, and the numeric shift $\tau$ passed to JD/RJD is offset from its true value by $0.012$, so several near-equally-plausible Ritz values compete to be closest to $\tau$. We run JD and RJD from $1000$ independent random starting vectors, correction equation solved exactly, tolerance $10^{-10}$, restart dimensions $m_{\min}=10$, $m_{\max}=20$, for each of standard, refined, harmonic, and harmonic-refined extraction. \Cref{table:extraction} reports the fraction of runs converging to the correct target eigenvalue, and the median outer iteration count among successful runs.
\begin{table}[!ht] 
	\caption{JD vs.\ RJD success rates and median outer iteration counts for an interior target, using four extraction methods, on a synthetic non-normal matrix with a densely clustered spectrum. Results are based on $1000$ independent random starting vectors for each method.}
	\label{table:extraction}
\centering
\begin{tabular}{lcccc}
\toprule
Extraction & \multicolumn{2}{c}{JD} & \multicolumn{2}{c}{RJD} \\
\cmidrule(lr){2-3} \cmidrule(lr){4-5}
method & success & med.\ iter & success & med.\ iter \\
\midrule
standard & 72.2\% & 22 & 79.4\% & 23 \\
refined & 63.1\% & 26 & 72.0\% & 27 \\
harmonic & 71.4\% & 30 & 81.8\% & 29 \\
harmonic-refined & 68.0\% & 34 & 78.1\% & 33 \\
\bottomrule
\end{tabular}
\end{table}

RJD achieves a higher success rate than JD for all four extraction methods, with very similar median outer iteration counts. We do not interpret this as evidence that RJD is intrinsically more reliable; rather, \Cref{table:extraction} indicates that sketching does not degrade reliability across the extraction strategies considered.

Harmonic and harmonic-refined extraction do not consistently outperform standard or refined extraction here, despite their usual advantage for interior eigenvalues in Arnoldi/Lanczos-type methods. This may be because harmonic Ritz approximations are more sensitive when the search subspace is still small and poorly adapted, while the JD correction equation already provides a shift-targeted expansion mechanism. We leave a fuller explanation to future work.

\section{Conclusion} \label{sec:conclusions} 
We introduced the randomized Jacobi-Davidson method that replaces the orthonormal search basis with a sketch-orthonormal one, uses a sketched Galerkin condition for eigenpair extraction, and modifies the correction equation accordingly. We proved that the method retains local quadratic convergence for non-Hermitian eigenvalue problems under a uniform separation condition (\Cref{thm:convRJD}), at roughly half the orthogonalization cost of the classical method. We also extended harmonic and refined extraction for interior eigenvalues (\Cref{sec:harmonicrefined}). Our numerical experiments confirm the predicted order of convergence and show reliability comparable to classical Jacobi-Davidson across real and synthetic problems, including with inexact correction solves and different extraction strategies for interior eigenvalues.

Several directions remain open. A two-sided formulation, tracking both the right and left eigenvectors, may yield cubic local convergence as in classical two-sided Jacobi-Davidson method~\cite{HochstenbachSleijpen03}. The framework could also be extended to generalized and nonlinear eigenvalue problems, and to partial Schur, JDQR-type methods~\cite{Fokkema1998} with sketch-orthogonal deflation. On the practical side, convergence theory for inexact correction solves would better justify the small GMRES budgets used in practice (\Cref{subsec:inexactce}), while mixed-precision implementation could further reduce computational cost.

\section*{Declaration of AI usage in the manuscript preparation process}
During the preparation of this work, the authors used Claude Opus 4.8 while refining and verifying the results in \Cref{sec:harmonicrefined}, ChatGPT-5.6 Sol Pro to double check the proofs and provide general assistance with the writing and editing, and Claude Code to assist with code generation and debugging. After using these tools and services, the authors reviewed and edited the content as needed and take full responsibility for the content of the published article.

\bibliographystyle{siamplain}
\bibliography{refs}

\end{document}